\documentclass{elsarticle}
\usepackage[T1]{fontenc} 
\usepackage{lmodern} 
\usepackage{amsmath}
\usepackage{amsthm}
\usepackage{amssymb}
\usepackage{ucs}
\usepackage{thmtools,mathtools}
\usepackage[hidelinks]{hyperref}
\usepackage{tikz-cd}
\usepackage{cleveref}
\usepackage{yhmath} 
\usepackage[italian,backgroundcolor=blue!5!white,textsize=scriptsize]{todonotes}
\usepackage[shortlabels]{enumitem}
\usepackage{xargs}
\usepackage{tikz}
\usepackage{geometry}
\usepackage{mleftright}
\newgeometry{margin=3.5cm}
\usepackage[normalem]{ulem}
\usepackage{cancel}

\declaretheorem[numberwithin=section,refname={Theorem,Theorems},Refname={Theorem,Theorems}, name=Theorem]{theorem}
\declaretheorem[sibling=theorem,style=remark,refname={Remark,Remarks}, Refname={Remark,Remarks},name=Remark]{remark}
\declaretheorem[sibling=theorem,refname={Proposition,Proposition},Refname={Proposition,Propositions},name=Proposition]{proposition}
\declaretheorem[sibling=theorem, name=Lemma, refname={Lemma,Lemmas},Refname={Lemma,Lemmas}]{lemma}
\declaretheorem[sibling=theorem,refname={Corollary,Corollaries},  Refname={Corollary,Corollaries},name=Corollary]{corollary}

\declaretheorem[sibling=theorem,style=definition,refname={Definition,Definitions},  Refname={Definition,Definitions},name=Definition]{definition}
\declaretheorem[sibling=theorem,style=definition,refname={Example,Examples},  Refname={Example,Examples},name=Example]{example}

\numberwithin{equation}{section}

\crefname{appendix}{Appendix}{Appendices}

\renewcommand{\leq}{\leqslant}
\renewcommand{\geq}{\geqslant}

\newcommand{\seq}{\subseteq}
\newcommand{\df}{\coloneqq}
\newcommand{\eps}{\varepsilon}

\newcommand{\succaux}[3]{\ensuremath{\langle#1\mid{#2\in#3}\rangle}}
\newcommandx*{\succustom}[3][2=n, 3 = \Np]{\succaux{#1}{#2}{#3}}

\newcommand{\Set}{\mathsf{Set}}
\newcommand{\Gr}{\mathsf{G}} 
\newcommand{\cG}{\overline{\Gr}} 
\newcommand{\KH}{\mathsf{KH}} 
\newcommand{\Kz}{\mathsf{K_a}} 
\newcommand{\KHz}{\mathsf{KH_a}}
\newcommand{\KHzop}{\mathsf{KH}_{\mathsf{a}}^\op}
\newcommand{\T}{{\mathsf{Top}}} 
\newcommand{\Tz}{{\mathsf{Top_a}}} 
\newcommand{\Alg}{\mathsf{Alg}_\tau}

\DeclareMathOperator{\Max}{Max}
\DeclareMathOperator{\C}{C_a}
\DeclareMathOperator{\Cc}{C}

\newcommand{\R}{\mathbb{R}}
\newcommand{\Q}{\mathbb{Q}}
\newcommand{\Z}{\mathbb{Z}}
\newcommand{\Zmod}[1][n]{\frac{1}{#1}\Z}
\newcommand{\N}{\mathbb{N}}
\newcommand{\Np}{\mathbb{N}^{+}}
\newcommand{\ULambda}{\rotatebox[origin=c]{180}{\ensuremath{\Lambda}}}

\DeclareMathOperator{\den}{den}
\DeclareMathOperator{\lcm}{lcm}
\DeclareMathOperator{\dist}{\mathtt{d}_{\ell}}
\newcommand{\lnorm}[1]{\lVert#1\rVert_{\ell}}
\newcommand{\lm}[1]{#1}
\DeclareMathOperator{\ev}{ev}
\DeclareMathOperator{\DIV}{div} 
\DeclareMathOperator{\ADC}{\mathtt{Z}} 
\newcommand{\AC}[1]{\wideparen{#1}} 
\newcommand{\down}[1]{#1^{\vartriangleleft}}
\newcommand{\up}[1]{#1^{\vartriangleright}}
\newcommand{\op}{\mathrm{op}}
\newcommand{\ua}{\mathord{\uparrow}}

\begin{document}
\begin{frontmatter}
\journal{arXiv}
\title{Stone-Gelfand duality for metrically complete lattice-ordered groups}
\author[add1]{Marco~Abbadini}
\ead{marco.abbadini.uni@gmail.it}
\author[add2]{Vincenzo~Marra\corref{cor}}
\ead{vincenzo.marra@unimi.it}
\author[add3]{Luca~Spada}
\ead{lspada@unisa.it}
\address[add1]{School of Computer Science,
University of Birmingham,
B15 2TT Birmingham, United Kingdom}
\address[add2]{Dipartimento di Matematica ``Federigo Enriques'', Universit\`a degli Studi di Milano, via Cesare Saldini 50, 20133 Milano, Italy}
\address[add3]{Dipartimento di Matematica, Universit\`a degli Studi di Salerno, Piazza Renato Caccioppoli, 2, 84084 Fisciano (SA), Italy}
\date{}

\cortext[cor]{Corresponding author}

\begin{abstract} 
    We extend Yosida's 1941 version of Stone-Gelfand duality to metrically complete unital lattice-ordered groups that are no longer required to be real vector spaces.  This calls for a generalised notion of compact Hausdorff space whose points carry an arithmetic character to be preserved by continuous maps.  The arithmetic character of a point is (the complete isomorphism invariant of) a metrically complete additive subgroup of the real numbers containing $1$---namely, either $\frac{1}{n}\mathbb{Z}$ for an integer $n = 1, 2, \dots$, or the whole of $\mathbb{R}$.  The main result needed to establish the extended duality theorem is a substantial generalisation of Urysohn's Lemma to such ``arithmetic'' compact Hausdorff spaces.  The original duality is obtained by considering the full subcategory of spaces every point of which is assigned the entire group of real numbers.  In the Introduction we indicate motivations from and connections with the theory of dimension groups.
\end{abstract}
\end{frontmatter}

\section{Introduction} \label{s:intro}
Let us write $\KH$ for the category of compact Hausdorff spaces and continuous maps. 
In the Thirties and Forties of the last century a number of mathematicians realised that the opposite category $\KH^{\op}$ can be represented, to within an equivalence, in several useful ways that variously relate to known mathematical structures of relevance in functional analysis. 
It is common to refer to this conglomerate of results by the umbrella term \textit{Stone-Gelfand Duality}.  
To provide a historically accurate picture of these theorems and of their mutual relationships would require far more space than we can afford here. 
Let us mention in passing that the authors involved include at least Gelfand and Kolmogorov \cite{GelfandKolmogoroff1939}, Gelfand and Neumark \cite{GelNeu1943}, the Krein brothers \cite{KreinKrein1940}, Kakutani \cite{Kakutani1941}, Yosida \cite{Yosida1941}, and Stone \cite{Stone1941}; and that the mathematical structures used in these representation theorems include at least rings of continuous functions, commutative $\mathrm{C}^*$-algebras, Kakutani's (M)-spaces, vector lattices (\textit{alias} Riesz spaces), and divisible Abelian lattice-groups.
In the present paper the more relevant structures are the ones used by Yosida, as follows.

Let us write $\mathsf{V}$ for the category whose objects are unital vector lattices---that is, lattice-ordered real vector spaces equipped with a distinguished unit---and whose morphisms are the linear lattice homomorphisms preserving the distinguished units. We recall that an element $1$ of the vector lattice $V$ is a \emph{\textup{(}strong order\textup{)} unit} if for any $v\in V$ there is a positive integer $p$ such that $v\leq p$. (In the latter expression, $p$ stands for $p\cdot 1$, and we shall use such shorthands without further warning.) Then $V$ can be equipped with a pseudometric induced by its unit $1$ by setting, for $v,w\in V$,
\begin{equation}\tag{*}\label{eq:metric}
\lVert v-w\rVert \coloneqq \inf{\left\{\lambda\in \mathbb{R} \mid \lambda \geq 0 \text{ and } \lambda \geq \lvert v-w \rvert \right\}},
\end{equation}
where $\lvert x \rvert \coloneqq x \vee (-x)$ is the absolute value of $x \in V$.  This pseudometric is a metric just when the vector lattice is Archimedean, i.e., if for all $x,y \in V$, whenever $px\leq y$ for all positive integers $p$, then $x \leq 0$. A unital vector lattice $V$ is then \emph{metrically complete} if the pseudometric induced by the unit $1$ is a metric, and $V$ is (Cauchy) complete in this metric. Thus, for instance, the pseudometric induced on the vector lattice $\Cc(X)$ of real-valued continuous functions on a compact Hausdorff space $X$ by the unit given by the function $X \to \mathbb{R}$ constantly equal to $1$ is the uniform metric, and $\Cc(X)$ is of course metrically complete.

Now, there is an adjunction between $\mathsf{V}$ and $\KH^{\op}$ that takes a unital vector lattice $V$ to its maximal spectral space, that is, its set of maximal (vector-lattice) ideals topologised by the hull-kernel topology; and a compact Hausdorff space $X$ to the unital vector lattice $\Cc(X)$. The adjunction descends to an equivalence between the full subcategory of $\mathsf{V}$ on the metrically complete vector lattices, and $\KH^{\op}$. This is the version of Stone-Gelfand Duality obtained by Yosida.

From a first perspective, in this paper we pursue the extension of Yosida's version of Stone-Gelfand Duality obtained by relaxing the hypothesis that $V$ be a real linear space, allowing $V$ to be instead a lattice-ordered Abelian group. Let us recall that a \emph{lattice-ordered group} (\emph{$\ell$-group}, for short) is a group endowed with a lattice structure such that the group operation distributes over both lattice operations; for background, please see \cite{MR0552653}.
The notions of unit and of pseudometric induced by the unit are defined for Abelian $\ell$-groups as in the vector lattice case---in order to use only the $\mathbb{Z}$-module structure of the underlying Abelian group in defining the pseudometric, replace the right-hand side of \eqref{eq:metric} by 
\[\tag{$\dagger$}\label{eq:metricbis}
\inf{\left\{\frac{p}{q}\in \mathbb{Q} \mid p\geq 0, q>0, \text{ and } p \geq q \lvert v-w \rvert \right\}}.
\]
Archimedean unital $\ell$-groups are defined as in the vector-lattice case; it is a standard result that Archimedean $\ell$-groups are Abelian \cite[Théorème 11.1.3]{MR0552653}, and it is again true that the pseudometric is a metric precisely when the Archimedean condition holds. A unital $\ell$-group is \emph{metrically complete} if it is Archimedean and complete in the metric induced by its unit. We abbreviate `unital lattice-ordered Archimedean (hence Abelian) group complete in the metric induced by its unit' by \emph{metrically complete $\ell$-group}. We thus have the category $\Gr$ of unital Abelian $\ell$-groups and lattice-group homomorphisms preserving the distinguished unit; and a full subcategory $\cG$ on those objects which are metrically complete. 

Then the question arises, \emph{what should replace the category $\KH$ in order to regain a dual equivalence with $\cG$}. 

To obtain some insight, consider the unit interval $[0,1] \subseteq \mathbb{R}$ and the $\ell$-group $\nabla$ of continuous piecewise-affine functions $[0,1]\to \mathbb{R}$ such that each affine piece $a(x)\coloneqq z_1x+z_2$ has integer coefficients $z_1,z_2 \in \mathbb{Z}$, with unit the function $[0,1]\to \mathbb{R}$ constantly equal to $1$. Then $\nabla$ is a subobject (i.e., a unital sublattice subgroup, known as unital $\ell$-subgroup) in $\Gr$ of $\Cc([0,1])$. The $\ell$-group $\nabla$ is not metrically complete. To compute its completion, let first $A_{x}\coloneqq\{z_1x+z_2\mid z_1,z_2\in\mathbb{Z}\}$ be the set of values attained by elements of $\nabla$ at the point $x\in[0,1]$. Next observe that  $A_{x}$ is a dense subgroup of $\mathbb{R}$ if $x$ is an irrational number, but it is the discrete subgroup $\frac{1}{d}\mathbb{Z}\coloneqq\{\frac{z}{d}\mid z\in \mathbb{Z}\}$ if $x$ is a rational number of denominator $d$. It follows that the metric completion of $\nabla$ is the subobject $\overline{\nabla}$ of $\Cc([0,1])$ consisting of those functions that, at each rational point of $[0,1]$ of denominator $d$, are constrained to take values in the subgroup $\frac{1}{d}\mathbb{Z}$.

The maximal spectral space of $\Cc([0,1])$---that is, its Stone-Gelfand-Yosida dual space---is just $[0,1]$. However, this is also the maximal spectral space of both $\nabla$ and $\overline{\nabla}$, because the functions in $\nabla$ separate points. In order to discern $\Cc([0,1])$ from $\overline{\nabla}$, their maximal spectral space $[0,1]$ can be enriched by information that specifies, for each $x\in [0,1]$, which subgroup of $\mathbb{R}$ the functions in the dual $\ell$-group can take values in. This can be summarised by the map $\den \colon [0,1]\to \mathbb{N}$ which assigns to a rational number $x\in [0,1]$ its denominator $d\in\mathbb{N}$, meaning that the set of possible values of functions at $x$ is $\frac{1}{d}\mathbb{Z}$; and to an irrational number $x\in [0,1]$ the value $0$, meaning that, at $x$, the values of the functions can range in the whole group $\mathbb{R}$. Here, the choice of $0$ is natural rather than conventional: it will transpire that a morphism, say, $\overline{\nabla}\to\overline{\nabla}$, contravariantly induces a continuous map $f\colon [0,1]\to [0,1]$ which \emph{decreases denominators} in the sense that it carries denominators to their divisors, i.e., for $x\in [0,1]$, $\den{f(x)}$ divides $\den{x}$; thus, there is no constraint on where such a map may send a point of denominator $0$.

What characteristic properties does the assignment $\den \colon [0,1]\to \mathbb{N}$ satisfy? We just pointed out that the dual continuous maps decrease denominators, and so the divisibility order of $\N$ is relevant; further inspection confirms that $\den \colon [0,1]\to \mathbb{N}$ is continuous when its codomain $\N$ is equipped with the \emph{upper topology} for its divisibility order---the topology whose closed sets are precisely the finite downsets, along with $\N$ itself. (Please see Section \ref{s:a.normal} for more details on the upper topology.) 

Abstracting from this example with the benefit of considerable hindsight, we define an \emph{arithmetic space} (\emph{a-space}, for short) to be a topological space equipped with a  \emph{denominator map}, a function to $\N$ that is continuous with respect to the upper topology induced by the divisibility order of the natural numbers. Morphisms between a-spaces are continuous maps that decrease denominators with respect to the divisibility order; we call them \emph{a-maps}. Fundamental examples of a-maps are (restrictions of) affine maps between finite-dimensional real vector spaces which descend to affine maps between the $\Z$-submodules of lattice (i.e., integer-coordinate) points with respect to chosen bases.

Now, it is clear that not every a-space can arise as the dual object of a metrically complete $\ell$-group; at the very least, the underlying space ought to be compact and Hausdorff. In Definition \Cref{d:anormal-space} we are able to identify the needed subclass of arithmetic spaces, which we call \emph{arithmetically normal}; as we shall prove, they provide the appropriate substitute for $\KH$. In other words, we prove the following generalisation of Yosida duality: The category of metrically complete $\ell$-groups is dually equivalent to the category of arithmetically normal spaces and a-maps between them.

Before proceeding, we should emphasise that the category of a-spaces admits a more conceptual presentation through lax comma categories; we provide some details on this intriguing connection in Remark \ref{r:lax} below.

\smallskip
From a second perspective, our main result is motivated by, and relates to, the theory of approximately finite-dimensional (AF) $\mathrm{C}^*$-algebras \cite{Effros} and their ordered $\mathrm{K}_0$ groups, known as dimension groups \cite{Goodbook}. In the remarks that follow, $\mathrm{C}^*$-algebras and dimension groups are always assumed to be unital. Elliott \cite{Elliott} showed dimension groups to be complete isomorphism invariants of AF $\mathrm{C}^*$-algebras. Elliott's result generalised previous well-known work by Glimm (1960) and Dixmier (1967), in turn rooted in the classical Murray-von Neumann approach to the classification of factors. Incidentally, we may now point out the significant conceptual differences between the two perspectives we are discussing in this Introduction. Let $X$ be a compact Hausdorff space. In the first perspective, elements of $\Cc(X)$ arise as observables of a classical---i.e., a commutative $\mathrm{C}^*$-algebra---system. In the second perspective, they arise as Murray-von Neumann ``dimensions'' of a sufficiently tame---i.e., an AF $\mathrm{C}^*$-algebra with lattice-ordered $\mathrm{K}_0$---non-classical system.

Effros, Handelman, and Shen proved (\cite[Theorem 2.2]{EHS}, see also \cite[Chapter 3]{Goodbook}) that dimension groups are exactly the directed, isolated---also known as ``unperforated''---unital partially ordered Abelian groups satisfying the Riesz interpolation property; any unital Abelian $\ell$-group is such. Goodearl and Handelman offered in \cite{GH} a systematic investigation of completions of a unital directed partially ordered Abelian groups with respect to the pseudometric induced on it by a state, i.e., a normalised positive group homomorphism to the reals. Within the theory they developed, Goodearl and Handelman obtained the following representation theorem for metrically complete unital Abelian $\ell$-groups.\footnote{As a terminological aside, we remark that Goodearl and Handelman called ``norm-complete'' those unital Archimedean $\ell$-groups which are complete in the metric \eqref{eq:metricbis} induced by their unit. Here, we quote their theorem using the terminology we introduced above, which adopts ``metrically complete'' rather than ``norm-complete''. Apart from this and from minor differences in notation, we quote \textit{verbatim}.}

\begin{theorem}[{Goodearl and Handelman, \cite[Theorem 5.5]{GH}}]\label{th:gh}Let $X$ be a compact Hausdorff space. For each $x \in X$, choose a subgroup $A_x$ of $\R$, so that $A_x\coloneqq\R$ or $A_x\coloneqq \frac{1}{n}\Z$ for some positive integer $n$. Set
\[
B\coloneqq\{p\in \Cc(X)\mid p(x) \in A_x \text{ for each } x \in X\},
\]
and give $B$ the order inherited from $\Cc(X)$. Then $B$ is an Archimedean metrically complete lattice-ordered Abelian group with unit. Conversely, any such group is isomorphic to one of this form.
\end{theorem}
\noindent We retain the notation of \cref{th:gh} in the remarks that follow. The maximal spectrum $Y$ of $B$ is always a continuous image of $X$, but $X$ and $Y$ may not be homeomorphic because $B$ may fail to separate the points of $X$. For instance, take $X=[0,1]$ and $A_x = \Z$ for each $x \in [0,1]$. Then $B$ is the set of constant functions from $[0,1]$ to $\Z$; thus, $B$ is $\Z$ and $Y$ is a singleton. Moreover, different assignments $\{A_x\}_{x \in X}$ on a space $X$ may give rise to the same $\ell$-group $B$. Consider for instance the one-point compactification of the discrete space $\N$, with the following two assignments.
The first assignment is constantly $\Z$; it yields an $\ell$-group that we call again $B$. The second assignment is everywhere $\Z$ except at the accumulation point, where it is $\frac{1}{2}\Z$; it yields an $\ell$-group that we call $B'$.
It is clear that $B \subseteq B'$.
The converse inclusion holds because the continuity of any $f \in B'$ forces the value of $f$ at the accumulation point to belong to $\Z$.

Theorem \ref{th:gh} provides a ``complete structural description of all'' \cite[p.\ 862]{GH} objects in $\cG$; the preceding paragraph points out that the relationship between $X$ and $B$ in that description is loose. In order to tighten it, we begin observing that the assignment $x \mapsto A_x$ makes $X$ into an a-space via the function $\zeta\colon X\to \N$ defined by $x\mapsto 0$ if $A_x$ is $\R$ and $x\mapsto n$ if $A_x$ is $\frac{1}{n}\Z$. We further observe that for any maximal ideal $\mathfrak{m}$ of $B$ the quotient $B/\mathfrak{m}$ has exactly one unital isomorphism to either $\R$ or $\frac{1}{n}\Z$, for a uniquely determined natural number $n$. Hence, the maximal spectrum $Y$ of $B$ becomes an a-space via the function $\zeta'\colon Y\to \N$ defined by $\mathfrak{m}\mapsto 0$ if $B/\mathfrak{m}$ is isomorphic to $\R$, and $\mathfrak{m}\mapsto n$ if $B/\mathfrak{m}$ is isomorphic to $\frac{1}{n}\Z$. To restate a part of our main results in relation to Theorem \ref{th:gh}: $(X, \zeta)$ is naturally isomorphic as an a-space to $(Y,\zeta')$ exactly when it is a normal a-space (\Cref{t:char fixed points on geom side}).
This is a consequence of the fact that arithmetically normal spaces satisfy the arithmetic version of Urysohn's Lemma given by \Cref{t:Urysohn's Lemma for a-spaces}.

To sum up, the generalisation of Yosida duality that we obtain here strengthens and extends \cref{th:gh}---without assuming it---to a duality for the category $\cG$.

\bigskip Either perspective sketched above promptly suggests much further research---for the sake of brevity, we do not elaborate. We do announce one further result driven by considerations of a rather different nature from the ones motivating this paper: To within an equivalence of categories, $\cG$ may be axiomatised by finitely many equations in an infinitary algebraic language with primitive operations of at most countably infinite arity; for classical Stone-Gelfand Duality, the corresponding result is the main one in \cite{MarraReggio}. The announced result will appear as a separate contribution.

\bigskip We now outline the structure of the paper with the aim of emphasising the proof strategy of our main result, \cref{t:MAIN}. The proof proceeds in a number of steps going from a more general dual adjunction to the duality of \cref{t:MAIN}. The first steps, readily flowing from general theory \cite{DimovTholen, PorstTholen}, are collected in \Cref{s:appendix}. 

The starting point is a dual adjunction between the categories $\Alg$ and $\Tz$ (\Cref{p:basicadj}), where $\Tz$ is the category of a-spaces, and $\Alg$ is the category of all structures in the algebraic signature $\tau \coloneqq\{+,\wedge,\vee,-,0,1\}$ of unital lattice-groups. This first dual adjunction is \emph{natural} in that it is induced by the \emph{dualising object} $\mathbb{R}$. 

The next step is to consider the full subcategory $\Gr$ of $\Alg$ on the unital Abelian $\ell$-groups, and the full subcategory $\Kz$ of $\Tz$ on the compact a-spaces. We prove in \Cref{t:restrictedbasicadj} that the first dual adjunction descends to a second one between these full subcategories. We note here that this second adjunction, unlike the first, is not natural with respect to the underlying-set functors on $\Gr$ and $\Kz$---$\R$ is not compact, and thus not an object of $\Kz$. It is the naturality of the first adjunction that makes it a convenient starting point for the proof of the second, non-natural adjunction.

Readers not interested in the details of these first two adjunctions can safely ignore \Cref{s:appendix} and begin directly from \Cref{s:dual-adjunction}. Here, we set the stage with the definition of a-space and recall as \Cref{p:specificadj} the dual adjunction between $\Gr$ and $\Kz$ constructed in \Cref{s:appendix}.

The body of the paper is then devoted to characterising the fixed subcategories of the dual adjunction between $\Gr$ and $\Kz$, i.e., the full subcategories on those objects at which the component of the (co)unit is an isomorphism.

Work begins in earnest in \Cref{s:fixedspaces}, where we prove that the compact a-spaces fixed by the adjunction are exactly the compact subspaces of $\R^I$ (for some set $I$), on which the denominator function is defined canonically (\Cref{t:easy char fixed points on geom side}) via the standard notion of denominator of a point of $\R^I$. 
The next aim is to obtain an intrinsic characterisation of such a-spaces as precisely the \emph{normal a-spaces} in the sense of \Cref{d:anormal-space}. 
In \Cref{s:a.normal} we prove that the compact subspaces of $\R^I$ indeed are normal a-spaces. 
The converse implication---that any normal a-space embeds as a subobject of $\R^I$---requires the material in Sections \ref{s:urysohn} and \ref{s:ury6}. The key result is a substantial generalisation of Urysohn's Lemma in which denominators are taken into account (\Cref{t:Urysohn's Lemma for a-spaces}).
The two implications are then assembled into the intrinsic characterisation in \Cref{t:char fixed points on geom side} at the end of \Cref{s:ury6}.

Concerning the objects of $\Gr$ fixed by the adjunction, these are characterised in \Cref{t:fixedalgebras} of \Cref{s:fixedalgebras}.  They are exactly the metrically complete $\ell$-groups. To achieve this result we use a strengthening of the Stone-Weierstrass Theorem for $\ell$-groups (\Cref{t:StWe_a-spaces}).

Finally, we state and prove our main result \Cref{t:MAIN}, and in \Cref{c:Yosida} we detail how it specialises to Yosida duality.

\section{Arithmetic spaces}\label{s:dual-adjunction}
We consider $\N\coloneqq\{0,1,2,\ldots\}$ with its divisibility order. For $n \in \N$ we write $\DIV{n}$ for the set of natural numbers that divide $n$. 
%
%
Recall that any (pre)ordered set $(X,\leq)$ can be equipped with the \emph{upper topology} \cite[p.\ 101]{Kelley}, which has as a subbasis for the closed sets the set of principal downsets $\operatorname{\downarrow}{x}\df\{y\in X\mid y\leq x\}$ for $x \in X\}$.
When we regard $\N$ as a topological space we always assume that it is endowed with the upper topology for the divisibility order. 

\begin{remark}\label{r:closed-of-N}
    The closed subsets for the upper topology on $\N$ are precisely the finite downsets along with $\N$.
    To see this, first note that the collection $\mathcal{F}$ consisting of these sets is closed under finite unions and arbitrary intersections.
    Next,  finite downsets are finite unions of principal downsets, so in any topology on $\N$ finite downsets are closed if principal ones are. Hence, any topology on $\N$ such that principal downsets are closed makes each member of $\mathcal{F}$ closed. 
    This shows that $\mathcal{F}$ indeed is the collection of closed sets of the upper topology on $\N$.
\end{remark}


\begin{definition} \label{d:a-space}
An \emph{a-space}---for arithmetic space---is a topological space $X$ equipped with a continuous function $\zeta \colon X \to \N$, where $\N$ is equipped with the upper topology for the divisibility order. We refer to $\zeta$ as the \emph{denominator function} of the a-space. We denote an a-space $(X, \zeta)$ by $X$ when the denominator function is understood.
A function $f \colon X \to X'$ between a-spaces $(X, \zeta)$ and $(X', \zeta')$ is said to \emph{decrease denominators} if, for each $x \in X$,
\begin{equation*}
	\zeta'(f(x)) \text{ divides } \zeta(x).
\end{equation*}
An \emph{a-map}  is a continuous map between a-spaces that decreases denominators. We write $\Tz$ for the category of a-spaces and a-maps. 
\end{definition}

\begin{remark}[A-spaces as a lax comma category]\label{r:lax}
    Any topological space $X$ carries the \emph{specialization preorder} defined by  $x \leq y$ if, and only if, every open subset of $X$ containing $x$ contains $y$. (See e.g.\ \cite[Definition 4.2.1]{Goubault-Larrecq2013}.)
    The specialization preorder induces a preorder on each hom-set in the category $\T$ of topological spaces and continuous maps---for maps $f,g \colon X \to Y$, one sets $f \leq g$ if, and only if,  $f(x) \leq g(x)$ holds for all $x \in X$. This makes $\T$ into a 2-category.
    
Now, the category $\Tz$ introduced in Definition \ref{d:a-space} above coincides with the lax comma category $\T//\N$ of topological spaces over $\N$. Please see \cite{ClementinoLucatelliLunes,ClementinoLucatelliLunesPrezado2024} for some recent results on lax comma ($2$-)categories, as well as further references and background. While $\T//\N$ is itself a $2$-category in a natural manner, in the rest of this remark only its $1$-dimensional structure is relevant.

Our main result Theorem \ref{t:MAIN} establishes that the category of metrically complete $\ell$-groups is dually equivalent to the full subcategory of $\T//\N$ on those objects $X\to\N$ such that $X$ is compact Hausdorff---hence normal---and the map correlates appropriately with the topological normality of $X$. We shall call such a-spaces \emph{arithmetically normal}; see Definition \ref{d:anormal-space} and Remark \ref{r:equivalent-to-normal} below. 

We point out that also the dual categories used in \cite{MR2038562,MR2996994}---two papers which may be regarded as precursors to our work here---can be presented as full subcategories of lax comma categories of spaces in an analogous way, even though they were not originally conceived of in that manner.
In \cite{MR2038562}, the authors prove that the category of locally finite MV-algebras is dually equivalent to the full subcategory determined by Stone spaces of $\T //\mathbf{G}$, where $\mathbf{G}$ is the poset of supernatural numbers equipped with the Scott topology. To within an isomorphism, $\mathbf{G}$ is the poset of (additive) subgroups of $\Q$ containing $1$.
The duality in \cite{MR2038562} was subsequently extended in \cite{MR2996994} to one for the category of the so-called \emph{weakly locally finite MV-algebras}.
The dual category can be presented as the full subcategory determined by Stone spaces of $\T//{\mathrm{Sub}\,\mathbb{R}}$, where ${\mathrm{Sub}\,\mathbb{R}}$ is the poset of subgroups of $\mathbb{R}$ containing $1$, equipped with the Scott topology. For comparison with the present context, in $\T//\N$ the space $\N$ is, to within an isomorphism, the poset of metrically complete subgroups of $\R$ containing $1$ equipped with the upper topology.
\end{remark}

\begin{definition}\label{d:real-denominator}
For $p/q \in \Q$, where $p \in \Z$ and $q \in \Np\coloneqq \N \setminus \{0\}$ are coprime, we define $\den{p/q} \df q$. For $r \in \R \setminus \Q$ we define $\den{r} \df 0$. We extend the definition of $\den$ to $\R^I$, for any set $I$, so that $(\R^I,\den)$ is the $I$-fold power of $(\R,\den)$ in the category of a-spaces and a-maps. For $p\df (p_{i})_{i \in I} \in \R^I$ we define 
\begin{equation} \label{eq:den}
	\den{p}\df \lcm{ \left\{ \den{p_{i}} \mid i \in I \right\} } , 
\end{equation}
where $\lcm$ stands for `least common multiple', and $0$ is the top element of $\N$ regarded as a lattice under the divisibility order. Notice that, for $p\notin \Q^I$, $\den{p} = 0$.
\end{definition}
\begin{lemma} \label{l:RI-is-aspace}
For every set $I$, the function $\den \colon \R^I\to \N$ is continuous.
Therefore, $(\R^{I},\den)$ is an a-space.
\end{lemma}
\begin{proof}
By the definition of the upper topology, we shall prove that $\den^{-1}[\DIV{n}]$ is closed for every $n \in \N$.
The case $n = 0$ is trivial because $\{ x \in \R^I \mid \den{x} \text{ divides }0 \} = \R^I$. If $n > 0$ then $\{ x \in \R^I \mid \den{x} \text{ divides }n \} = \left(\Zmod\right)^I$, and this set is closed because $\Zmod$ is closed and the product of closed subsets is closed in the product topology. Therefore, $\den^{-1}[\DIV{n}]$ is a closed subset of $\R^I$.
\end{proof}
If $S \seq \R^I$ is any subset, we write $(S,\den)$ for the a-space whose denominator function $\den \colon S\to \N$ is the restriction to $S$ of $\den \colon \R^I \to \N$. When no confusion can arise, we simply write $S$ in place of $(S,\den)$.

We denote the category of compact a-spaces and a-maps with $\Kz$ and the category of unital Abelian $\ell$-groups with unit-preserving lattice-group homomorphisms with $\Gr$.
In the remainder of this section we establish an adjunction between $\Kz$ and $\Gr$, confining the proofs to \cref{s:appendix}.

\begin{definition}
For an object $X$ in $\Kz$, we write $\C(X)$ for the set of a-maps from $X$ to $\R$, the latter seen as an a-space. The set $\C(X)$ can be endowed with an $\ell$-group structure by defining operations pointwise. Since $X$ is compact, the functions in $\C(X)$ are bounded, and hence the function constantly equal to $1$ is a unit.
This makes $\C(X)$ into a unital Abelian $\ell$-group.
\end{definition}

\begin{definition}\label{d:Max}
If $G$ is an object in $\Gr$, we define an a-space $(\Max{G},\zeta)$ as follows:
\begin{enumerate}
    \item\label{d:Max:1} $\Max{G}$ is the set of unital $\ell$-homomorphisms from $G$ into $\R$.
    \item\label{d:Max:2} The topology on $\Max{G}$ is generated by the sets
    \[
    \{x \in \Max{G}\mid x(g) \in U\},
    \]
    for $g \in G$, and $U$ an open subset of $\R$.
    \item\label{d:Max:3} The denominator function is given by
    \[
    \zeta(x) \df \lcm\{\den{x(g)} \mid g \in G\},
    \]
    for $x \in \Max{G}$.
\end{enumerate}
\end{definition}

\begin{remark}\label{r:MaxG-is-compact}
Observe that $\Max{G}\seq \R^{G}$.  The topology defined in \Cref{d:Max}\eqref{d:Max:2} is the topology induced on $\Max{G}$ by the product topology on $\R^{G}$ and the denominator defined in \Cref{d:Max}\eqref{d:Max:3} is just the restriction to $\Max{G}$ of the denominator of $(\R^{G},\den)$ (see \Cref{d:real-denominator}). Finally, as shown in \Cref{t:restrictedbasicadj}, for every unital Abelian $\ell$-group $G$, $\Max{G}$ is a compact a-space. 
\end{remark}

The assignments $\Max$ and $\C$ are extended to contravariant functors between $\Gr$ and $\Kz$ by defining their action on morphisms via pre-composition.

For any compact a-space $X$ and any $x\in X$, the function $\ev_x \colon \C(X) \to \R$,
defined by $\ev_x(a) \df a(x)$, is called the \emph{evaluation at $x$}.
For any unital Abelian $\ell$-group $G$ and any $g \in G$, the function $\ev_g \colon \Max{G} \longrightarrow \R$, defined by $\ev_g(x)\df x(g)$, is called the \emph{evaluation at $g$}.
We postpone to \cref{s:appendix} the proof of the following statement (see \cref{t:restrictedbasicadj}).

\begin{proposition} \label{p:specificadj}
The contravariant functors $\Max \colon \Gr \to \Kz$ and $\C \colon \Kz\to \Gr$ are adjoint, with units given, for $G \in \Gr$ and $X \in \Kz$, by the evaluations
\begin{align*}
    	\eta_X \colon X&\longrightarrow \Max\C(X) & \varepsilon_G \colon G & \longrightarrow \C(\Max{G})\\
	x&\longmapsto \ev_x \colon \C(X) \to \R, & g&\longmapsto \ev_g \colon \Max{G}\to \R.\\
\end{align*}
\end{proposition}

\section{A first characterisation of the a-spaces fixed by the adjunction} \label{s:fixedspaces}

In this section, we provide a first characterisation of the compact a-spaces fixed by the adjunction (\Cref{t:easy char fixed points on geom side}) as the ones that are isomorphic as a-spaces to $(K, \den)$ for some set $I$ and some compact subset $K$ of $\R^I$.
Eventually, in \Cref{t:char fixed points on geom side} we will obtain an intrinsic characterisation of these objects.

The next lemma is well known in several variants. For instance, for rings of continuous functions it goes back to Gelfand and Kolmogorov \cite[Lemma to Theorem I]{GelfandKolmogoroff1939}, who in turn cite the ground-breaking 1937 paper \cite{StoneBA} where Stone used a different structure. We need a (known) version for unital lattice-groups for which we are not aware of a suitable reference; we therefore include a proof. For $X$ a compact topological space we write $\Cc(X)$ for the $\ell$-group of continuous real-valued functions on $X$ with unit the function constantly equal to $1$. If $X$ is an a-space then $\C(X)$ is a unital $\ell$-subgroup of $\Cc(X)$.
\begin{lemma} \label{l:ev is surjective}
	Let $X$ be a compact space and let $G$ be a unital $\ell$-subgroup of $\Cc(X)$. Suppose that for all distinct $x,y \in X$ there is $f \in G$ with $f(x) \neq f(y)$. Let $\varphi \colon G \to \R$ be an $\ell$-homomorphism that preserves $1$. Then there exists a unique $x \in X$ such that $\varphi$ is the evaluation at $x$ (i.e., for all $f \in G$, $\varphi(f) = f(x)$).
\end{lemma}
\begin{proof}
    Uniqueness follows from the fact that $G$ separates the points of $X$.
    
    Let us prove existence.
    Set 
    \[
    A \coloneqq \{ x \in X \mid \forall f \in G \text{ if }\varphi(f) \leq 0 \text{ then } f(x) \leq 0\}.
    \]
    We prove $A \neq \emptyset$.
    Since $X$ is compact and $A$ is the intersection of the family of closed sets $\{ x \in X \mid f(x) \leq 0\}$ for $f \in G$ such that $\varphi(f) \leq 0$, it is enough to show that this family has the finite intersection property. 
    Let $f_1, \dots, f_n \in \{f \in G \mid \varphi(f) \leq 0\}$ and suppose, by way of contradiction,
    \begin{equation} \label{eq:FIP}
        \bigcap_{i \in \{1, \dots, n\}} \{ x \in X \mid f_i(x) \leq 0\} = \emptyset.
    \end{equation}
    Since $\varphi$ preserves $0$ and $1$, $G$ is not a singleton, whence $X \neq \emptyset$.
    By \cref{eq:FIP} and $X \neq \emptyset$, $n \neq 0$.
    Since $n \neq 0$, it makes sense to consider the function $f_1 \lor \dots \lor f_n \colon X \to \R$. Since this is a continuous function on a nonempty compact space, it has a minimum $\lambda$. 
    By \cref{eq:FIP}, for all $x \in X$ and $i \in \{1, \dots, n\}$ we have $f_i(x) > 0$; it follows that $\lambda > 0$.
    Therefore, there is $k \in \N$ such that $k \lambda \geq 1$.
    Thus, $1 \leq k \lambda \leq k(f_1 \lor \dots \lor f_n)$, and applying $\varphi$ we obtain
    \[
    1 \leq k(\varphi(f_1) \lor \dots \lor \varphi(f_n)) \leq k (0 + \dots + 0) = 0,
    \]
    a contradiction.
    Therefore, $A \neq \emptyset$.

    Choose $x \in A$.
    Fix $f \in G$, and let us prove $\varphi(f) = f(x)$.
    For all $q \in \Np$ and $p \in \Z$ the condition $\varphi(f) \leq \frac{p}{q}$ implies
    $\varphi(qf - p) \leq 0$, which implies $q f(x) - p \leq 0$ since $x \in A$, which implies $f(x) \leq \frac{p}{q}$.
    Therefore, $f(x) \leq \varphi(f)$.
    Analogously, $(-f)(x) \leq \varphi(-f)$, i.e.\ $\varphi(f) \leq f(x)$.
    Thus, $\varphi(f) = f(x)$.
\end{proof}

\begin{theorem} \label{t:easy char fixed points on geom side}
	The following conditions are equivalent for a compact a-space $(X,\zeta)$.
	\begin{enumerate}
		\item\label{t:easy char:item1} There exist a set $I$ and a compact subset $K$ of $\R^I$ such that $(X,\zeta)$ and $(K,\den)$ are i\-so\-mor\-phic a-spaces.
		\item\label{t:easy char:item2} There exists a unital Abelian $\ell$-group $(G,1)$ such that $(X,\zeta)$ and $(\Max{G},\den)$ are isomorphic a-spaces.
		\item\label{t:easy char:item3} The map
		\begin{align*}
			\eta_X \colon X&\longrightarrow \Max{\C(X)} \\
			x&\longmapsto \ev_x \colon \C(X) \to \R
		\end{align*} is an isomorphism of a-spaces.
	\end{enumerate}
\end{theorem}

\begin{proof}
	The implication \ref{t:easy char:item3} $\Rightarrow$ \ref{t:easy char:item2} is trivial and the implication \ref{t:easy char:item2} $\Rightarrow$ \ref{t:easy char:item1} follows from \Cref{r:MaxG-is-compact}.
	
	Let us prove the implication \ref{t:easy char:item1} $\Rightarrow$ \ref{t:easy char:item3}. Without loss of generality, we may identify $X$ with $K$. The family of projections $\{ \pi_i \colon \R^I\to \R \}_{i \in I}$ separates the points of $X$; therefore, by \Cref{l:ev is surjective}, the map $\eta_X \colon X \to\Max{\C(X)}$ is a bijection. Since $X$ and $\Max{\C(X)}$ are compact Hausdorff spaces and the map $\eta_X$ is continuous (as a consequence of \cref{p:basicadj}),
	$\eta_X$ is closed by the Closed Map Lemma. Hence $\eta_X$ is a homeomorphism.
	
	To conclude the proof, we show that $\eta_X$ preserves denominators. By hypothesis
	\[\zeta(x) = \lcm\{ \den{x_i}\mid i \in I \} \]
	and by definition
	\[\zeta(\eta_X(x)) = \lcm\{ \den{f(x)}\mid f \in \C(X) \}.\] 
	Since each $f \in \C(X)$ is an a-map, $\zeta(\eta_X(x))$ divides $\zeta(x)$ for any $x \in X$. Furthermore, for every $i \in I$ we have $\pi_i \in \C(X)$, and thus \[
	\{ \den{x_i}\mid i \in I \} = \{ \den{\pi_i(x)} \mid i \in I \} \seq \{ \den{f(x)} \mid f \in \C(X) \}.
	\]Therefore, $\lcm \{ \den x_i \mid i \in I \}$ divides $\lcm\{ \den{f(x)} \mid f \in \C(X) \}$, i.e., $\zeta(x)$ divides $\zeta(\eta_X(x))$. Therefore, $\zeta(\eta_X(x)) = \zeta(x)$.
\end{proof}

\section{Normality for a-spaces} \label{s:a.normal}

In this section, and in the following two, we develop the machinery needed to prove Theorem \ref{t:char fixed points on geom side}. In that theorem, we will characterise the compact a-spaces fixed by the adjunction as precisely those satisfying the appropriate generalisation of normality to the arithmetic context.

\begin{definition} \label{d:anormal-space}
    An \emph{arithmetically normal space} (\emph{normal a-space}, for short) is an a-space $(X,\zeta)$ with the following properties.
    \begin{enumerate}[label = N\arabic*., ref = N\arabic*]
    \item\label{d:anormal-space:item1} $X$ is compact.
    \item\label{d:anormal-space:item3prime} For any two distinct points $x$ and $y$ of $X$, there exist two disjoint open neighbourhoods $U$ and $V$ of $x$ and $y$, respectively, such that, for every $z \in X \setminus (U \cup V)$, $\zeta(z) = 0$ (see \Cref{figure:separationpoints}).
    \begin{figure}[h!]
            \begin{center}
            \begin{tikzpicture}[scale=0.5]
            \def\universe{(-2,-2) rectangle (9,3.5)}
            \def\firstcircle{(1,0.75) circle (2cm)}
            \def\secondcircle{(6,0.75) circle (2cm)}
            
            \tikzset{filled/.style={fill=blue!20, draw=blue!50, thick},
            open/.style={fill=white, dashed, thick},
            }
            \begin{scope}
            \draw[filled] \universe node[below left=0.2] {$X$};
            \clip \firstcircle \secondcircle;
            \end{scope}
            
            \draw[open] \firstcircle node[above=0.55] {$U$};
            \draw[open] \secondcircle node[above=0.55] {$V$};
            \node at (1,0.75) {\phantom{$x$} {$\bullet$} $x$};
            \node at (6,0.75) {\phantom{$y$} {$\bullet$} $y$};
            \node[draw=blue!50] at (3.5,-1.3) {$\zeta=0$};
            \end{tikzpicture}
            \end{center}
            \caption{The `arithmetically Hausdorff' property \ref{d:anormal-space:item3prime}.} \label{figure:separationpoints}
        \end{figure}
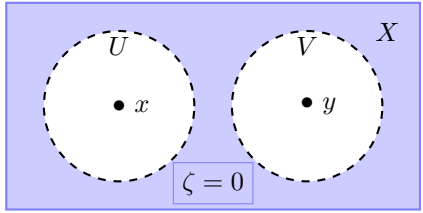
    \end{enumerate}

    We write $\KHz$ for the full subcategory of $\Kz$ on the normal a-spaces.
\end{definition}

\begin{remark} \label{r:equivalent-to-normal}
    The name `arithmetically normal space' stems from the observation, proved in the next proposition, that in the presence of \ref{d:anormal-space:item1} the `arithmetically Hausdorff' condition \ref{d:anormal-space:item3prime} can be equivalently replaced by
    \begin{enumerate}[label = N\arabic*'., ref = N\arabic*',start=2]
    \item\label{d:anormal-space:item3} $X$ is Hausdorff, and for any two disjoint closed subsets $A$ and $B$ of $X$, there exist two disjoint open sets $U$ and $V$ containing respectively $A$ and $B$ and such that, for every $x \in X \setminus (U \cup V)$, $\zeta(x) = 0$.
    \end{enumerate}
    %
        
%
\noindent This also explains our notation $\KHz$ for the category of normal a-spaces.
\end{remark}

\begin{proposition} \label{p:N3-equiv-N3'}
A compact  a-space satisfies \ref{d:anormal-space:item3prime} if and only if it satisfies \ref{d:anormal-space:item3}.
\end{proposition}

\begin{proof}
Since points are closed in any Hausdorff space, \ref{d:anormal-space:item3} implies \ref{d:anormal-space:item3prime}.

For the converse implication, suppose the a-space $(X,\zeta)$ satisfies \ref{d:anormal-space:item3prime}.  We first prove that, for any closed subset $C\seq X$ and any $x\in X$ with $x\not\in C$, there exist two disjoint open subsets $O_1$ and $O_2$ such that $C\seq O_1$, $x\in O_2$, and $\zeta(z)=0$ for all $z\in X \setminus (O_1\cup O_2)$.  Since $(X,\zeta)$ satisfies \ref{d:anormal-space:item3prime}, for any $y\in C$ there exists $O_{1,y},O_{2,y}$ disjoint open neighbourhoods of $y$ and $x$, respectively, such that $\zeta(z)=0$ for all $z\in X \setminus (O_{1,y}\cup O_{2,y})$. The family $\{O_{1,y}\mid y\in C\}$ covers $C$, which is closed, whence compact.  Thus, there exist $y_1,\dots,y_n\in C$ such that $C\seq O_{1,y_1}\cup\dots\cup O_{1,y_n}$.  Set $O_1\df O_{1,y_1}\cup\dots\cup O_{1,y_n}$ and $O_2\df O_{2,y_1}\cap\dots\cap O_{2,y_n}$.  The sets $O_1$ and $O_2$ are disjoint open subsets of $X$ such that $C\seq O_1$ and $x\in O_2$.  If $z\in X \setminus (O_1\cup O_2)$ then $z\not\in O_2$, which means that there exists $k\leq n$ with $z\not\in O_{2,y_k}$. Since $z\not\in O_1$, we have $z \not\in O_{1,y_k}$.  Therefore, $z\in X \setminus (O_{1,y_k}\cup O_{2,y_k})$, which implies $\zeta(z)=0$.

Let us now prove the statement.  Let $C_1,C_2$ be disjoint closed sets. We just proved that, for any $y\in C_1$, there exist two disjoint open subsets $O_{1,y}$ and $O_{2,y}$ such that $C_2\seq O_{2,y}$, $y\in O_{1,y}$, and $\zeta(z)=0$ for all $z\in X \setminus (O_{1,y}\cup O_{2,y})$. The family $\{O_{1,y}\mid y\in C_1\}$ covers $C_1$, which is compact.  Thus, there exist $y_1,\dots,y_n\in C_1$ such that $C_1\seq O_{1,y_1}\cup\dots\cup O_{1,y_n}$.  Set $O_1\df O_{1,y_1}\cup\dots\cup O_{1,y_n}$ and $O_2\df O_{2,y_1}\cap\dots\cap O_{2,y_n}$.  The sets $O_1$ and $O_2$ are disjoint open subsets of $X$ such that $C_1\seq O_1$ and $C_2\seq O_2$.  If $z\in X \setminus (O_1\cup O_2)$ then $z\not\in O_2$, which means that there exists $k\leq n$ with $z\not\in O_{2,y_k}$. Since $z\not\in O_1$, we have $z\not\in O_{1,y_k}$.  Therefore we have $z\in X \setminus (O_{1,y_k}\cup O_{2,y_k})$, which implies $\zeta(z)=0$.
\end{proof}

\begin{example}
The a-space $([0,1],{\den})$ is a normal a-space.  
To check  \ref{d:anormal-space:item3prime} holds, consider two points $a<b$ in $[0,1]$.  By the density of the irrational numbers, there exists $r \in [0,1] \setminus \Q$ such that $a<r<b$.  Hence, the open sets $[0,r)$ and $(r,1]$ separate $a$ from $b$, $\{r\} = X \setminus \left([0,r) \cup (r,1]\right)$, and $\den{r} = 0$.
\end{example}

\Cref{t:char fixed points on geom side} below shows that normality characterises, up to a-isomorphism, the a-spaces of the form $(K, \den)$ for $K$ a compact subset of $\R^{I}$, $I$ an arbitrary set.  We conclude this section by proving the easy implication, i.e.\ that any such a-space is normal.  For the other, much deeper implication, we will need to prove a form of Urysohn's Lemma for normal a-spaces. This we do in the following two sections.

\begin{theorem} \label{t:compact is a-normal}
For any set $I$ and any compact subspace $K \seq \R^I$, $(K,\den)$ is a normal a-space.
\end{theorem}
\begin{proof}


To show  \ref{d:anormal-space:item3prime} holds, let $x \neq y \in K$. 
There exists $i \in I$ such that $\pi_i(x) \neq \pi_i(y)$, where $\pi_i \colon K \to \R$ is the projection on the $i$-th coordinate. Let $\lambda$ be an irrational number between $\pi_i(x)$ and $\pi_i(y)$. The projection $\pi_i$ is continuous, so the subsets $\pi_i^{-1}[(-\infty,\lambda)]$ and $\pi_i^{-1}[(\lambda, + \infty)]$ are open. Since $\den{\lambda} = 0$ and $\pi_i$ decreases denominators, the denominator of each point of $K$ outside $\pi_i^{-1}[(-\infty,\lambda)] \cup \pi_i^{-1}[(\lambda, + \infty)]$ is zero.
\end{proof}

\section{Drafts of a-maps} \label{s:urysohn}

The aim of the following two sections is to prove a version of Urysohn's Lemma for a-spaces.  As in the classical case, one crucially needs to construct a function into $[0,1]$ with a prescribed behaviour. The definition of an a-map combines a topological and an arithmetical condition; the next definition helps isolate the sole arithmetical constraint: a-maps decrease denominators. The points which a denominator decreasing map can send into a fixed $r\in\R$ are bound to have denominators which are divided by $\den{r}$. Accordingly, for any subset $\Lambda\seq \R$ we consider the set of \emph{admissible} denominators (in the above sense) of points which can be sent into $\Lambda$.

\begin{definition} \label{d:adc}
For $\Lambda \seq \R$, we define 
\[
    \ADC(\Lambda) \df \left\{ n \in \N \mid \text{There exists } \lambda \in \Lambda \text{ such that } \den{\lambda} \text{ divides } n \right\}.
\]
\end{definition}

\begin{remark} \label{l:closed under}
The following properties are elementary.
\begin{enumerate}
    \item \label{item:closed-under-multiples} $\ADC(\Lambda)$ is closed under multiples.
    \item \label{item:closed-under-divisors} $\N\setminus \ADC(\Lambda)$ is closed under divisors.
    \item \label{l:monotonicity of ADC 0} The operator $\ADC$ is monotone: if $\Lambda_0\seq \Lambda_1\seq \R$ then $\ADC(\Lambda_0) \seq \ADC(\Lambda_1)$.
    \item If $(\Lambda_i)_{i \in I}$ is a family of subsets of $\R$, then
\[\ADC\left(\bigcup_{i \in I}\Lambda_i\right) = \bigcup_{i \in I}\ADC(\Lambda_i).\]
\end{enumerate}
\end{remark}

In general, $\ADC$ does not commute with intersections. We identify a case in which it does as a corollary of the next lemma. We say  a poset is \emph{lower-directed} if it is nonempty, and every pair of elements has a common lower bound. If $P$ is a preorder and $A\seq P$ then $\ua A\coloneqq \{p \in P \mid a \leq p \text{ for some } a \in A\}$ is the \emph{up-closure} of $A$ in $P$. 
\begin{lemma} \label{l:Esakia's-lemma}
    Let $ f\colon X \to Y$ be a continuous function between topological spaces, and suppose  $X$ is Hausdorff.
    Let $(K_i)_{i \in I}$ be a lower-directed family of compact subsets of $X$.
    Then
    \[
    \ua f\mleft[\bigcap_{i \in I}K_i \mright] = \bigcap_{i \in I} \ua f[K_i],
    \]
    where the up-closures are computed with respect to the specialization preorder of $Y$.
\end{lemma}
\begin{proof}
    This is essentially a special case of \cite[Lemma~8.1]{EscardoLawsonEtAl2004}, but we include a short proof.
    The left-to-right inclusion is obvious.
    For the right-to-left inclusion, note that the up-closure of a set $A$ coincides with the intersection of all open sets containing $A$.
    Therefore, it is enough to prove  $\bigcap_{i \in I} \ua f[K_i]$ is contained in every open set containing $\ua f[\bigcap_{i \in I}K_i]$.
    Let $U$ be an open subset of $Y$ containing $\ua f[\bigcap_{i \in I}K_i]$.
    From the inclusion $f[\bigcap_{i \in I}K_i] \subseteq U$ we deduce $\bigcap_{i \in I}K_i \subseteq f^{-1}[U]$.
    By \cite[Exercise~4.4.18]{Goubault-Larrecq2013} (i.e., the fact that every Hausdorff space is well-filtered), there is $j \in I$ such that $K_j \subseteq f^{-1}[U]$. Therefore, $f[K_j] \subseteq U$, which implies $\ua f[K_j] \subseteq U$. Thus, $\bigcap_{i \in I} f[K_i] \subseteq {\ua f[K_j]} \subseteq U$.
\end{proof}

Applying \cref{l:Esakia's-lemma} to the continuous function $\den \colon \R \to \N$, we obtain:
\color{black}

\begin{corollary}\label{l:very nice property}
For every lower-directed family $(\Lambda_i)_{i \in I}$ of compact subsets of $\R$,
\[
    \ADC\left(\bigcap_{i \in I}\Lambda_i\right) = \bigcap_{i \in I}\ADC(\Lambda_i).
\]
\end{corollary}

\begin{lemma} \label{l:cofinite}
For any $\alpha<\beta \in \R$, the set $\ADC([\alpha,\beta])$ is cofinite in $\N$.
\end{lemma}

\begin{proof}
Let $k \in \Np$ be such that $\frac{1}{k} \leq \beta - \alpha$, or, equivalently, $k \geq \frac{1}{\beta - \alpha}$. Then, for every $n\geq k$, there exists $\gamma \in \Zmod \cap [\alpha,\beta]$. In other words, $\gamma \in [\alpha,\beta]$ and $\den{\gamma}$ divides $n$.
Therefore $n \in \ADC([\alpha,\beta])$.
\end{proof}

Given an a-space $X$ and an arbitrary point $r\in\R$, we consider the set of points of $X$ whose denominator is a multiple of $\den{r}$.  Those are exactly the points of $X$ that can be sent into $r$ by a denominator decreasing map. More generally, for any subset $\Lambda \seq \R$, we consider the set of points of $X$ which can be sent into $\Lambda$ by some denominator decreasing map.

\begin{definition} \label{d:ac}
For $(X,\zeta)$ an a-space and $\Lambda \seq \R$, we define 
\[
\AC{\Lambda}_X \df \left\{ x \in X \mid \text{There exists } \lambda \in \Lambda \text{ such that } \den{\lambda} \text{ divides } \zeta(x) \right\}.
\]
We write $\AC{\Lambda}$ when $X$ is understood.
\end{definition}

\begin{remark} \label{r:ACandADC}
Notice that $\AC{\Lambda}_X = \zeta^{-1}[\ADC(\Lambda)]$.
\end{remark}

Let $X$ be an a-space, let $\alpha \leq \beta \in \R$, and let $f \colon X \to [\alpha, \beta]$ be an a-map. In the next definition we single out the main properties of the sets
\[
    f^{-1}\big[[\alpha,r]\big] \text{ and } f^{-1}\big[[r, \beta]\big]
\]
as $r$ ranges in a subset of $[\alpha,\beta]$. This can be seen as a generalisation of the concept of \emph{decomposition} of a topological space \cite[Section 3-6, p.\ 132]{hocking1988topology}.

\begin{definition} \label{d:draft}
Let $X$ be an a-space, and let $\alpha \leq \beta \in \R$. An \emph{$[\alpha,\beta]$-draft} on $X$ consists of a subset $D \seq [\alpha,\beta]$ containing $\alpha$ and $\beta$ and two closed subsets $\down{r}$ and $\up{r}$ of $X$ for each $r \in D$, with the following properties.
\begin{enumerate}[label = D\arabic*., ref = D\arabic*]
\item\label{d:draft:item1} $\up{\alpha} = \down{\beta} = X$.
\item\label{d:draft:item2} For all $r \in D$, $\down{r} \cup \up{r} = X$.
\item\label{d:draft:item3} For all $r < s$ in $D$, $\down{r} \cap \up{s} = \emptyset$.
\item\label{d:draft:item4} For all $r \leq s$ in $D$, $\up{r} \cap \down{s} \seq \AC{[r,s]}$.
\end{enumerate}
\end{definition} 

Given real numbers $\alpha \leq \beta$, a subset $D$ of $[\alpha,\beta]$ containing $\alpha$ and $\beta$, an a-space $X$ and an a-map $f \colon X \to [\alpha, \beta]$, it is easily seen that $(D,(f^{-1}[[\alpha,r]],f^{-1}[[r, \beta]])_{r \in D})$ is an $[\alpha,\beta]$-draft.

\begin{lemma} \label{l:draft-inclusion-properties}
Let $(D, (\down{d}, \up{d})_{d \in D})$ be an $[\alpha,\beta]$-draft on $X$.
Given two elements $r\leq s$ in $D$, we have
\[
    \down{r} \seq \down{s}, \qquad \up{s} \seq \up{r}, \qquad \text {and } \qquad \down{s} \cup \up{r} = X.
\]
\end{lemma}

\begin{proof}
If $r = s$, then clearly $\down{r} \seq \down{s}$ and $\up{s} \seq \up{r}$. If $r<s$, then 
\begin{align*}
\down{r}& = \down{r} \cap X = \down{r} \cap (\down{s} \cup \up{s}) & \text{ by \ref{d:draft:item2} in \Cref{d:draft}}\\
& = (\down{r} \cap \down{s}) \cup (\down{r} \cap \up{s}) & \\
& = (\down{r} \cap \down{s}) \cup \emptyset & \text{ by \ref{d:draft:item3} in \Cref{d:draft}}\\
& = \down{r} \cap \down{s}.&
\end{align*}
 So we conclude that $\down{r} \seq \down{s}$.  The second inclusion is proved similarly.  Finally, the last equality in the statement holds because, by the first inclusion in the statement, $\down{r} \cup \up{r}\seq \down{s} \cup \up{r}$, and the former is equal to $X$ by \ref{d:draft:item2} in \Cref{d:draft}.
\end{proof}

\begin{definition}
Let $f \df (D,(\down{d},\up{d})_{d \in D})$ and $f' \df (D',(d^{\vartriangleleft'},d^{\vartriangleright'})_{d \in D'})$ be $[\alpha,\beta]$-drafts on $X$. We say that $f'$ \emph{refines} $f$ if $D \seq D'$ and, for all $d \in D$, $\down{d} = d^{\vartriangleleft'}$ and $\up{d} = d^{\vartriangleright'}$.
\end{definition}

\begin{lemma} \label{l:refinement by one}
Let $X$ be a normal a-space, let $\alpha\leq\beta \in \R$, and let $(D,(\down{d},\up{d})_{d \in D})$ be an $[\alpha,\beta]$-draft on $X$. Let $a< b$ such that $[a,b] \cap D = \{a,b\}$ and let $\lambda \in(a,b)$. Then there exist two closed subsets $\down{\lambda}$, $\up{\lambda}$ of $X$ so that $(D \cup \{\lambda\},(\down{d},\up{d})_{d \in D \cup \{\lambda\}})$ is a $[\alpha,\beta]$-draft on $X$ that refines $(D,(\down{d},\up{d})_{d \in D})$.
\end{lemma}

\begin{proof}
To obtain the sets $\down{\lambda}$ and $\up{\lambda}$ claimed in the statement we  apply the separation property \ref{d:anormal-space:item3} in \Cref{r:equivalent-to-normal} to appropriate  subsets of $X$.  We start by defining those subsets and proving that they are closed and disjoint.
Let 
\begin{align}
\label{eq:d-A-and-B}
A\df \down{a} \cup [ (\up{a} \cap \down{b}) \setminus \AC{[\lambda,b]} ] \qquad\text{ and }\qquad B\df \up{b} \cup [ (\up{a} \cap \down{b}) \setminus \AC{[a,\lambda]} ].
\end{align}
To prove  $A$ is closed, we observe  $A = \down{a} \cup [(\up{a} \cap \down{b}) \cap (X \setminus \AC{[\lambda,b]}) ]$ and  $\down{a}$, $\up{a}$ and $\down{b}$ are closed in $X$ by \Cref{d:draft}. So it is enough to prove that $X \setminus \AC{[\lambda,b]}$ is closed, too.
By \Cref{r:ACandADC}, we have
\[
X \setminus \AC{[\lambda,b]} = X \setminus \zeta^{-1}[\ADC([\lambda,b])] = \zeta^{-1}[\N \setminus \ADC([\lambda,b])].
\]
The set $\N \setminus \ADC([\lambda,b])$ is finite by \Cref{l:cofinite}, and it is closed under divisors by \cref{item:closed-under-divisors} in \Cref{l:closed under}.
Thus, $\N \setminus \ADC([\lambda,b])$ is closed in the upper topology (see \cref{r:closed-of-N}), and by the continuity of $\zeta \colon X \to \N$ the set $X \setminus \AC{[\lambda,b]} = \zeta^{-1}[\N \setminus \ADC([\lambda,b])]$ is closed. We conclude  $A$ is closed. The proof that $B$ is closed is analogous.

Now we prove  $A$ and $B$ are disjoint.	Applying distributivity,
\begin{align*}
A \cap B & = \left(\down{a} \cup \left((\up{a} \cap \down{b}) \setminus \AC{[\lambda,b]} \right) \right) \cap \left(\up{b} \cup \left( (\up{a} \cap \down{b}) \setminus \AC{[a,\lambda]}\right) \right) = \\
& = \left(\down{a} \cap \up{b}\right) \;\cup\; \left(\down{a} \cap \left((\up{a} \cap \down{b}) \setminus \AC{[a,\lambda]}\right) \right) \;\cup\\
& \cup\; \left(\left((\up{a} \cap \down{b}) \setminus \AC{[\lambda,b]}\right) \cap \up{b} \right) \;\cup\; \left(\left((\up{a} \cap \down{b}) \setminus \AC{[\lambda,b]}\right) \cap \left( (\up{a} \cap \down{b}) \setminus \AC{[a,\lambda]} \right) \right).
\end{align*}
We prove  each of the four sets in the union above is empty.
\begin{enumerate}
\item $\down{a} \cap \up{b} = \emptyset$ by definition of draft.
\item $\begin{aligned}[t]
\down{a} \cap \big((\up{a} \cap \down{b}) \setminus \AC{[a,\lambda]}\big) & \seq \down{a} \cap (\up{a} \setminus \AC{[a,\lambda]})&\\
& = (\down{a} \cap \up{a}) \setminus \AC{[a,\lambda]}&\\
& \seq \AC{\{a\}}\setminus \AC{[a,\lambda]}&\text{by \ref{d:draft:item4}}\\
& = \emptyset & \text{since }\AC{\{a\}} \seq \AC{[a,\lambda]}.
\end{aligned}$
\item Analogously, $((\up{a} \cap \down{b}) \setminus \AC{[\lambda,b]}) \cap \up{b} = \emptyset$.
\item $\begin{aligned}[t]
&\big((\up{a} \cap \down{b}) \setminus \AC{[\lambda,b]}\big) \cap \big((\up{a} \cap \down{b}) \setminus \AC{[a,\lambda]}\big) \\
& = (\up{a} \cap \down{b}) \setminus \left(\AC{[\lambda,b]} \cup \AC{[a,\lambda]}\right)&\text{} \\
& = (\up{a} \cap \down{b}) \setminus \left(\AC{ [a,\lambda] \cup [\lambda,b]}\right) & \text{} \\
& = (\up{a} \cap \down{b}) \setminus \AC{[a,b]} & \text{} \\
& = \emptyset&\text{ since } \up{a} \cap \down{b} \seq \AC{[a,b]}.
\end{aligned}$
\end{enumerate}
We conclude  $A$ and $B$ are disjoint closed subsets of $X$.  Therefore, by \ref{d:anormal-space:item3} in \Cref{r:equivalent-to-normal}, there exist two open disjoint sets $U$ and $V$ containing respectively $A$ and $B$ and such that, for every $x \in X \setminus (U \cup V)$, $\zeta(x) = 0$. 

We set $\down{\lambda}\df X \setminus V$ and $\up{\lambda}\df X \setminus U$ and we prove  $\left(D \cup \{\lambda\},(\down{d},\up{d})_{d \in D \cup \{\lambda\}}\right)$ is an $[\alpha,\beta]$-draft on $X$. 
By construction, $\down{\lambda}$ and $\up{\lambda}$ are closed.  Condition \ref{d:draft:item1} in the definition of $[\alpha,\beta]$-draft continues to hold because the extremes $\alpha$ and $\beta$ have not changed. Furthermore, since $U \cap V = \emptyset$, we have $\down{\lambda} \cup \up{\lambda} = X$, so also \ref{d:draft:item2} is satisfied.
To prove \ref{d:draft:item3}, from the definitions we immediately deduce
\begin{align} \label{eq:up-and-down1}
\down{a} \seq A \seq \down{\lambda} \text{ and } \up{b} \seq B \seq \up{\lambda}.
\end{align}
Let $r \in D$ with $r<\lambda$. Since by hypothesis $[a,b] \cap D = \{a,b\}$ and $\lambda \in(a,b)$, we have $r\leq a$, and thus, by \Cref{l:draft-inclusion-properties}, $\down{r} \seq \down{a}$. By \cref{eq:up-and-down1}, $\down{a} \seq A$. Hence $\up{\lambda} \cap \down{r} \seq \up{\lambda} \cap \down{a} \seq \up{\lambda} \cap A = \emptyset$, where the latter holds because $\up{\lambda}$ is the complement of $U\supseteq A$. Analogously for $r > \lambda$.

To conclude the proof, we check  \ref{d:draft:item4} holds.
First, observe that 
\begin{align} \label{eq:up-and-down2}
\down{\lambda} \seq \down{b}\text{ and }\up{\lambda} \seq \up{a}.
\end{align}
Indeed, by definition of $\down{\lambda}$, we have $\down{\lambda} \cap B = \emptyset$.  So, by \cref{eq:up-and-down1}, $\down{\lambda} \cap \up{b} = \emptyset$; since $\down{b} \cup \up{b} = X$ by \ref{d:draft:item2} in \cref{d:draft}, we have $\down{\lambda} \seq \down{b}$. The inclusion $\up{\lambda} \seq \up{a}$ is proved analogously.

Let $r \in D$ with $r\leq \lambda$. If $r = \lambda$, then 
\begin{align*}
\up{\lambda} \cap \down{\lambda}&\seq \zeta^{-1}[\{0\}] & \text{since, for every }x \in \up{\lambda} \cap \down{\lambda} = X \setminus (U \cup V),\,\zeta(x) = 0\\
&\seq \zeta^{-1}[\ADC(\{\lambda\})] & \text{because $0 \in \ADC(\{\lambda\})$ by \cref{item:closed-under-multiples} in \Cref{l:closed under}}\\
& = \AC{\{\lambda\}} & \text{by \Cref{r:ACandADC}}. 
\end{align*}

We now consider the case $r<\lambda$. By hypothesis $r\leq a$, whence $\up{a} \seq \up{r}$ by \Cref{l:draft-inclusion-properties}. Observe that:
\begin{align*}
\up{r} \cap \down{\lambda}& = \up{r} \cap (\up{r} \cup \down{\lambda}) \cap X \cap \down{\lambda}& \text{because }\up{r} \cap \down{\lambda} \seq X \\
& = (\up{r} \cup \up{a}) \cap (\up{r} \cup \down{\lambda}) \cap (\down{a} \cup \up{a}) \cap (\down{a} \cup \down{\lambda})&\text{using $\up{a} \seq \up{r}$, \ref{d:draft:item2}, and \ref{eq:up-and-down1}}\\
& = (\up{r} \cap \down{a}) \cup (\up{a} \cap \down{\lambda}) & \text{by distributivity.}
\end{align*}
As a consequence, to prove that $\up{r} \cap \down{\lambda} \seq \AC{[r,\lambda]}$ it is enough to show that $\up{r} \cap \down{a} \seq \AC{[r,\lambda]}$ and $\up{a} \cap \down{\lambda} \seq \AC{[r,\lambda]}$.  For the first inclusion, since $r\leq a$, by \ref{d:draft:item4} we have $\up{r} \cap \down{a} \seq \AC{[r,a]} \seq \AC{[r,\lambda]}$. For the other inclusion, we compute
\begin{align*}
\emptyset& = \up{a} \cap (\down{\lambda} \cap B) &\text{because }\down{\lambda} \cap B = \emptyset \\ 
& = (\up{a} \cap \down{\lambda}) \cap \left(\up{b} \cup \left( (\up{a} \cap \down{b}) \setminus \AC{[a,\lambda]} \right) \right) &\text{by the definition of $B$ (\cref{eq:d-A-and-B})}\\
& = (\up{a} \cap \down{\lambda} \cap \up{b}) \cup \left((\up{a} \cap \down{\lambda}) \cap \left( (\up{a} \cap \down{b}) \setminus \AC{[a,\lambda]}\right) \right) &\text{by distributivity} \\
& = (\up{a} \cap \down{\lambda} \cap \up{b}) \cup \left((\up{a} \cap \down{\lambda}) \setminus \AC{[a,\lambda]} \right) &\text{by \cref{eq:up-and-down2}} \\
& = \emptyset \cup ((\up{a} \cap \down{\lambda}) \setminus \AC{[a,\lambda]})&\text{because, by \ref{d:draft:item3}, } \down{\lambda} \cap \up{b} = \emptyset.
\end{align*}
We conclude  $(\up{a} \cap \down{\lambda}) \setminus \AC{[a,\lambda]} = \emptyset$, which implies $\up{a} \cap \down{\lambda} \seq \AC{[a,\lambda]}$. Since $r\leq a$, the latter entails $\up{a} \cap \down{\lambda} \seq \AC{[r,\lambda]}$. This completes the proof that \ref{d:draft:item4} holds in the case $r < \lambda$.
The proof that \ref{d:draft:item4} holds also for any $r \in D$ such that $r > \lambda$ is analogous.
\end{proof}

\begin{lemma} \label{l:extension to dense}
Let $X$ be a normal a-space, let $\alpha \leq \beta \in \R$, and let $(D,(\down{d},\up{d})_{d \in D})$ be an $[\alpha,\beta]$-draft on $X$. If $D$ is closed, there exists an $[\alpha,\beta]$-draft $(D',(\down{d},\up{d})_{d \in D'})$ on $X$ that refines $(D,(\down{d},\up{d})_{d \in D})$ and such that $D'$ is dense in $[\alpha,\beta]$.
\end{lemma}

\begin{proof}
Let $q_n$ be a dense sequence in $[\alpha, \beta]$. 
Define $D_{0} \df D$ and $D_{n+1} \df D_{n} \cup \{ q_{n} \}$. Notice that, for all $n \in \N$, $D_{n}$ is a closed subset of $\R$.  Define a sequence $\Delta_n$ of $[ \alpha, \beta ]$-drafts as follows.
If $q_{n}$ already belongs to $D_{n}$, set $\Delta_{n+1} \df \Delta_n$.  Otherwise, set 
\[
a \coloneqq \max{ ( D_n \cap (-\infty, q_n] ) } \text{ and } b \coloneqq \min{ ( D_n \cap [q_n, +\infty) ) }.
\] 
The numbers $a$ and $b$ exist because the sets $D_n \cap (-\infty, q_n]$ and $D_n \cap [q_n,+\infty)$ are nonempty, closed and bounded subsets of $\R$. Moreover, $a\neq q_n$ and $b\neq q_n$ because $q_n\notin D_n$. It follows that $q_n \in (a,b)$ and $[a,b] \cap D_n = \{a,b\}$, so \Cref{l:refinement by one} applies and we define $\Delta_{n+1} \df(D_{n+1},(\down{d},\up{d})_{d \in D_{n+1}})$ accordingly.

The sequence $\Delta_n$ of $[\alpha,\beta]$-drafts is a chain with respect to the refinement partial order.  Therefore, if we let $D' \df \bigcup_{n \in \N}D_n$, then $(D',(\down{d},\up{d})_{d \in D'})$ is an $[\alpha,\beta]$-draft on $X$ that refines $\Delta_n$ for each $n \in \N$. Since $D'$ contains the sequence $q_n$, it is dense in $[\alpha, \beta]$.
\end{proof}

\section{The arithmetic Urysohn Lemma, and the a-spaces fixed by the adjunction} \label{s:ury6}

In this section we will show that given any $[\alpha,\beta]$-draft on $X$ defined on a dense $D\seq[\alpha,\beta]$, one can construct an a-map $f\colon X\to [\alpha,\beta]$ with the property that $f[\down{r}]\seq [\alpha,r]\text { and }f[\up{r}] \seq [r,\beta]$.  This result will be used to prove the arithmetic version of Urysohn's Lemma.

\begin{lemma} \label{claim:Urysohn1}
Let $(D,(\down{d},\up{d})_{d \in D})$ be an $[\alpha,\beta]$-draft on $X$ and suppose that $D$ is dense in $[\alpha,\beta]$.  Let $\lambda \in [\alpha,\beta]$.
The following two conditions are equivalent.
\begin{enumerate}
    \item \label{i:1}
    For all $r \in D$ if $r < \lambda$ then $x \in \up{r}$.
    \item \label{i:2}
    For all $r \in D$ if $r < \lambda$ then $x \notin \down{r}$.
\end{enumerate}
Moreover, the following two conditions are equivalent.
\begin{enumerate}[resume]
    \item \label{i:3}
    For all $r \in D$ if $\lambda > r$ then $x \in \down{r}$.
    \item \label{i:4}
    For all $r \in D$ if $\lambda > r$ then $x \notin \up{r}$.
\end{enumerate}
\end{lemma}

\begin{proof}
    Let us prove the implication \ref{i:1} $\Rightarrow$ \ref{i:2}.
    Let $r \in D$ with $r < \lambda$.
    Then, since $D$ is dense in $[\alpha, \beta]$, there exists $s \in D$ such that $r < s < \lambda$.
    By hypothesis, since $s<\lambda$, we have $x \in \up{s}$.
    By \ref{d:draft:item3} in \cref{d:draft}, we have $\down{r} \cap \up{s} = \emptyset$, whence $x \notin \down{r}$.
    
    The implication \ref{i:2} $\Rightarrow$ \ref{i:1} follows from the fact that, by \ref{d:draft:item3} in \cref{d:draft}, for every $r \in D$ we have $\down{r} \cup \up{r} = X$.
    
    The equivalence between \ref{i:3} and \ref{i:4} can be proved similarly.
\end{proof}

\begin{lemma} \label{l:inf-is-sup}
Let $X$ be an a-space and let $(D,(\down{d},\up{d})_{d \in D})$ be an $[\alpha,\beta]$-draft on $X$. If $D$ is dense in $[\alpha,\beta]$, then for every $x \in X$ we have
\[
\inf \{ r \in D \mid x \in \down{r} \} = \sup \{ r \in D \mid x \in \up{r} \}.
\]
\end{lemma}

\begin{proof}
Fix $x \in X$ and $\lambda \in [\alpha,\beta]$. Observe that $\lambda$ is a lower bound of $\{ r \in D \mid x \in \down{r} \}$ if and only if
\begin{align} \label{l:inf-is-sup:eq1}
\forall r \in D,\; r < \lambda \text{ implies }x \notin \down{r}.
\end{align}
(Indeed, the contrapositive of \eqref{l:inf-is-sup:eq1} states that $\lambda$ is such a lower bound.)

Next, we show that $\lambda\geq s$ for any lower bound $s$ of $\{ t \in D \mid x \in \down{t} \}$ if and only if
\begin{align} \label{l:inf-is-sup:eq2}
\forall r \in D,\; r > \lambda \text{ implies } x \in \down{r}.
\end{align}
Indeed, assume $\lambda$ is such. It follows that any $r > \lambda$ cannot be a lower bound of $\{ t \in D \mid x \in \down{t} \}$, and so there is $s$ with $x \in \down{s}$ and $s<r$; by \Cref{{l:draft-inclusion-properties}} we conclude that $x \in \down{s} \seq \down{r}$.
Vice versa, assume \eqref{l:inf-is-sup:eq2} and suppose $s$ is a lower bound for $\{ t \in D \mid x \in \down{t} \}$. Arguing by contradiction, if $s > \lambda$, then by the density of $D$ in $[\alpha,\beta]$ there exists $u$ with $\lambda<u<s$. By \eqref{l:inf-is-sup:eq2}, we have $x \in \down{u}$, and this contradicts the fact that $s$ is a lower bound. 

To sum up, $\lambda = \inf \{ r \in D \mid x \in \down{r} \}$ if and only if \eqref{l:inf-is-sup:eq1} and \eqref{l:inf-is-sup:eq2} hold.
By \Cref{claim:Urysohn1}, conditions \eqref{l:inf-is-sup:eq1} and \eqref{l:inf-is-sup:eq2} are equivalent to `for all $r \in D$ if $r < \lambda$ then $x \in \up{r}$' and `for all $r \in D$ if $r > \lambda$ then $x \notin \up{r}$', respectively. In turn, the latter conditions together are equivalent to $\lambda = \sup \{ r \in D \mid x \in \up{r} \}$.
We conclude that $\inf \{ r \in D \mid x \in \down{r} \} = \sup \{ r \in D \mid x \in \up{r} \}$. 
\end{proof}
\begin{definition} \label{def:realisation}
Let $X$ be an a-space, let $\alpha \leq \beta \in \R$, and let $(D,(\down{d},\up{d})_{d \in D})$ be an $[\alpha,\beta]$-draft on $X$. A \emph{realisation} of $(D,(\down{d},\up{d})_{d \in D})$ is a function $f \colon X \to [\alpha,\beta]$ such that, for every $r \in D$, 
\[
f[\down{r}]\seq [\alpha,r]\text { and }f[\up{r}] \seq [r,\beta].
\]
\end{definition}
\begin{theorem} \label{t:dense has realisation}
Let $X$ be an a-space, let $\alpha \leq \beta \in \R$, and let $(D,(\down{d},\up{d})_{d \in D})$ be an $[\alpha,\beta]$-draft on $X$, with $D$ a dense subset of $[\alpha,\beta]$.	Then $(D,(\down{d},\up{d})_{d \in D})$ has a unique realisation, given by
\begin{align*}
f \colon X & \longrightarrow [\alpha,\beta]\\
x & \longmapsto \inf \{ r \in D \mid x \in \down{r} \} = \sup \{ r \in D \mid x \in \up{r} \}.
\end{align*}
Furthermore, $f$ is continuous and decreases denominators.
\end{theorem}
\begin{proof}
By \Cref{l:inf-is-sup}, the definitions using $\inf$ or $\sup$ are equivalent.
Let us first prove that $f$ is a realisation of $(D,(\down{d},\up{d})_{d \in D})$.  Let $r \in D$. For every $x \in \down{r}$ we have $f(x) = \inf \{ s \in D \mid x \in \down{s} \} \leq r$. Hence, $f[\down{r}] \seq [\alpha,r]$. Similarly, for every $x \in \up{r}$ we have $f(x) = \sup \{ s \in D \mid x \in \up{s} \} \geq r$.  Hence, $f[\up{r}] \seq [r,\beta]$. 

Let us now prove that $f$ is continuous. Let $\lambda \in [\alpha,\beta]$.  It suffices to prove that $f^{-1}[[\alpha,\lambda]]$ and $f^{-1}[[\lambda,\beta]]$ are closed.
We have
\begin{align*}
f^{-1} \big[[\alpha,\lambda]\big] & = \left\{ x \in X \mid \sup \{ r \in D \mid x \in \up{r} \} \leq \lambda \right\}\\
& = \left\{ x \in X \mid \forall r \in D \cap (\lambda,\beta]\ \  x \notin \up{r} \right\}\\
& = \left\{ x \in X \mid \forall r \in D \cap (\lambda,\beta]\ \  x \in \down{r} \right\} & \text{by \Cref{claim:Urysohn1}}\\
& = \bigcap_{r \in D \cap (\lambda,\beta]} \down{r},
\end{align*} 
which is closed.  Similarly,
\begin{align*}
f^{-1} \big[[\lambda,\beta]\big] & = \left\{ x \in X \mid \inf \{ r \in D \mid x \in \down{r} \} \geq \lambda \right\}\\
& = \left\{ x \in X \mid \forall r \in D \cap [\alpha,\lambda) \ \  x \notin \down{r} \right\}\\
& = \left\{ x \in X \mid \forall r \in D \cap [\alpha,\lambda) \ \  x \in \up{r} \right\} & \text{by \Cref{claim:Urysohn1}}\\
& = \bigcap_{r \in D \cap [\alpha,\lambda)} \up{r},
\end{align*}
which is closed, too.

Next, we prove that $f$ decreases denominators. Fix $x \in X$. Set 
\[
\Lambda_x \df \{ r \in D \mid x \in \up{r} \} \text{ and } \ULambda_x \df \{ r \in D \mid x \in \down{r} \}.
\]
By \ref{d:draft:item1}, $\alpha \in \Lambda_x$ and $\beta \in \ULambda_x$. Set $\Omega \df \{ [s,t] \mid s \in \Lambda_x, t \in \ULambda_x \}$. By the definition of $f$, $f(x) = \sup{\Lambda_x} = \inf{\ULambda_x}$; therefore, $\bigcap_{I \in \Omega} I = \{ f(x) \}$. We verify that $\Omega$ meets the hypotheses of \Cref{l:very nice property}. First, $\Omega\neq \emptyset$ because $[\alpha,\beta]\in \Omega$. Furthermore, for every $[s_1,t_1], [s_2,t_2]\in \Omega$, there exists $[s,t]\in \Omega$ such that $[s,t]\seq [s_1,t_1] \cap [s_2,t_2]$: consider $s \df \max \{ s_1,s_2 \}$ and $t\df \min \{ t_1,t_2 \}$. By \Cref{l:very nice property} we obtain 
\[
\bigcap_{I \in \Omega} \ADC(I) = \ADC\left(\bigcap_{I \in \Omega}I\right) = \ADC(\{ f(x) \}).
\]
If $[s,t] \in \Omega$, then $x \in \up{s} \cap \down{t}$. Hence, by \ref{d:draft:item4}, $x \in \AC{[s,t]}$. Thus, by \Cref{r:ACandADC}, $\zeta(x) \in \ADC([s,t])$. It follows that $\zeta(x) \in \bigcap_{I \in \Omega} \ADC(I) = \ADC(\{ f(x) \})$. This amounts to saying that $\den f(x)$ divides $\zeta(x)$.
Therefore, $f$ decreases denominators.

Finally, let us prove that the realisation of $(D,(\down{d},\up{d})_{d \in D})$ is unique.
Let $f'$ be a realisation of $(D,(\down{d},\up{d})_{d \in D})$.
Fix $x \in X$.
For all $r \in D$ such that $x \in \down{r}$, we have $f'(x) \leq r$. Therefore, $f'(x) \leq \inf \{ r \in D \mid x \in \down{r} \}$. For all $r \in D$ such that $x \in \up{r}$, we have $f'(x) \geq r$. Thus, $ f'(x) \geq \sup \{ r \in D \mid x \in \up{r} \}$.
By \cref{l:inf-is-sup}, $\inf \{ r \in D \mid x \in \down{r} \} = \sup \{ r \in D \mid x \in \up{r} \}$, and hence
\[
    f'(x) = \inf \{ r \in D \mid x \in \down{r} \} = \sup \{ r \in D \mid x \in \up{r} \}. \qedhere
\]
\end{proof}

\begin{theorem} \label{t:strong urysohn}
Let $X$ be a normal a-space, let $\alpha\leq \beta \in \R$ and let $(D,(\down{d},\up{d})_{d \in D})$ be an $[\alpha,\beta]$-draft on $X$, with $D$ a closed subset of $[\alpha,\beta]$. Then $(D,(\down{d},\up{d})_{d \in D})$ admits a continuous denominator decreasing realisation $f \colon X \to [\alpha,\beta]$.
\end{theorem}

\begin{proof}
By \Cref{l:extension to dense}, there is dense subset $D'$ of $[\alpha, \beta]$ and an $[\alpha,\beta]$-draft $(D',(\down{d},\up{d})_{d \in D'})$ on $X$ that refines $(D,(\down{d},\up{d})_{d \in D})$. By \Cref{t:dense has realisation}, $(D',(\down{d},\up{d})_{d \in D'})$ has a continuous denominator decreasing realisation $f \colon X \to [\alpha,\beta]$. 
Clearly, $f$ is a realisation of $(D,(\down{d},\up{d})_{d \in D})$, too.
\end{proof}

\begin{remark}
Not every draft admits a continuous realisation: the hypothesis that $D$ is closed in \Cref{t:strong urysohn} cannot be dropped.
Indeed, consider the normal a-space $X \df [0, \frac{1}{2}]$ equipped with the Euclidean topology and with the function $\zeta \colon X \to \N$ constantly equal to $0$.
This is a normal a-space.
Set $D\df[0,\frac{1}{2}) \cup \{ 1 \}$; for every $r \in [0,\frac{1}{2})$, set $\down{r} \df [0,r]$, and $\up{r} \df[r,\frac{1}{2}]$; set $\down{1} = [0,\frac{1}{2}]$ and $\up{1} = \{ \frac{1}{2} \}$. This gives a $[0,1]$-draft on $X$. However, there is no continuous realisation of $(D,(\down{d},\up{d})_{d \in D})$. Indeed, if $f$ were a continuous realisation of $(D,(\down{d},\up{d})_{d \in D})$, then we would have $f(\frac{1}{2}) = 1$ and, for every $x \in [0,\frac{1}{2})$, $f(x) \leq \frac{1}{2}$: a contradiction.
\end{remark}

\begin{theorem}[Urysohn's Lemma for a-spaces]\label{t:Urysohn's Lemma for a-spaces}
Let $A$ and $B$ be disjoint closed subsets of a normal a-space $(X, \zeta_X)$, and let $\alpha,\beta \in \R$ with $\alpha \leq \beta$. Suppose that the following conditions hold.
\begin{enumerate}
    \item \label{i:one}
    For all $x \in X$ there is $\lambda \in [\alpha, \beta]$ such that $\den{\lambda}$ divides $\zeta_X(x)$.
    \item \label{i:two}
    For all $a \in A$, $\den{\alpha}$ divides $\zeta_X(a)$.
    \item \label{i:three}
    For all $b \in B$, $\den{\beta}$ divides $\zeta_X(b)$.
\end{enumerate}
Then there exists an a-map $f \colon X \to [\alpha, \beta]$ such that $f(x) = \alpha$ for all $x \in A$, and $f(x) = \beta$ for all $x \in B$.
\end{theorem}

\begin{proof}
Set $\down{\alpha} \df A$, $\up{\alpha} \df X$, $\down{\beta} \df X$, and $\up{\beta} \df B$. In light of our assumptions \ref{i:one}--\ref{i:three},
$(\{ \alpha,\beta \}, (\down{d},\up{d})_{d \in \{ \alpha,\beta \}})$ is an $[\alpha,\beta]$-draft on $X$. 
By \Cref{t:strong urysohn}, $(\{ \alpha,\beta \}, (\down{d},\up{d})_{d \in \{ \alpha,\beta \}})$ admits a continuous denominator decreasing realisation $f \colon X \to [\alpha, \beta]$.
By definition of realisation (\cref{def:realisation}), the following conditions hold: (i) for all $x \in A = \down{\alpha}$ we have $f(x) \in \{\alpha\}$, and (ii) for all $x \in B = \up{\beta}$ we have $f(x) \in \{\beta\}$.
\end{proof}

\begin{remark}
The classical Urysohn Lemma is obtained from \Cref{t:Urysohn's Lemma for a-spaces} upon choosing $\zeta_X$ to be the function constantly equal to $0$.
\end{remark}

\begin{corollary} \label{t:eta in injective}
    Given a normal a-space $X$ and given $x \neq y \in X$, there exists an a-map $f \colon X \to [0,1]$ such that $f(x) = 0$ and $f(y) = 1$.
\end{corollary}

\begin{proof}
    It is immediate that $\AC{[0,1]} = \AC{\{0\}} = \AC{\{1\}} = X$.
    Thus, to complete the proof we may apply \Cref{t:Urysohn's Lemma for a-spaces} with $\alpha = 0$, $\beta = 1$, $A = \{x\}$ and $B = \{y\}$.
\end{proof}

\begin{corollary} \label{t:eta preserves denominators}
    Given a normal a-space $X$ and given $x \in X$, there exists an a-map $f \colon X \to [0,1]$ such that $\den{f(x)} = \zeta(x)$.
\end{corollary}

\begin{proof}
    Choose $\lambda \in (0,1]$ such that $\den{\lambda} = \zeta(x)$.
    Since $0 \in [0, \lambda]$, we have $X = \AC{[0, \lambda]}$.
    Thus, we may apply Urysohn's lemma (\Cref{t:Urysohn's Lemma for a-spaces}) with $\alpha = 0$, $\beta = \lambda$, $A = \emptyset$ and $B = \{x\}$.
\end{proof}

\begin{theorem} \label{t:char fixed points on geom side}
    The following conditions are equivalent for a compact a-space $(X,\zeta)$.
    \begin{enumerate}
        \item\label{i:a-normal} $X$ is arithmetically normal.
        \item\label{i:compact} There exist a set $I$ and a compact $K \seq \R^I$ such that $(X,\zeta)$ and $(K,\den)$ are isomorphic a-spaces.
        \item\label{i:closed} There exist a set $I$ and a closed $K \seq [0,1]^I$ such that $(X,\zeta)$ and $(K,\den)$ are isomorphic a-spaces.
    \end{enumerate}
\end{theorem}

\begin{proof}
    The implication \ref{i:closed} $\Rightarrow$ \ref{i:compact} is obvious.
    The implication \ref{i:compact} $\Rightarrow$ \ref{i:a-normal} is given by \Cref{t:compact is a-normal}.
    We prove the implication \ref{i:a-normal} $\Rightarrow$ \ref{i:closed}.
    Consider the set $I$ of all a-maps from $X$ to $[0,1]$.
    By the universal property of the product, there exists a unique a-map $g \colon X \to [0,1]^I$ that factors all maps in $I$ through the projections.
    By \Cref{t:eta in injective}, for all $x \neq y$ there exists an a-map $f \colon X \to [0,1]$ such that $f(x) \neq f(y)$.
    This means that $g$ is injective.
    Since $g$ is a continuous map between compact Hausdorff spaces, $g$ is closed. Hence $g$ is a homeomorphism onto its image.
    By \Cref{t:eta preserves denominators}, for every $x \in X$ there exists an a-map $f \colon X \to [0,1]$ such that $\den{f(x)} = \zeta(x)$.
    Therefore, $\den{g(x)} = \zeta(x)$.
    In conclusion, $g$ is an isomorphism of a-spaces onto its image.
\end{proof}

\begin{corollary} \label{c:fixed-spaces}
    A compact a-space $(X,\zeta)$ is arithmetically normal if and only if the unit $\eta_X \colon X \to \Max{\C(X)}$ of the adjunction in \cref{p:specificadj} is an isomorphism of a-spaces.
\end{corollary}

\begin{proof}
    This follows from the equivalence \ref{i:a-normal} $\Leftrightarrow$ \ref{i:compact} in \cref{t:char fixed points on geom side} and the equivalence \ref{t:easy char:item1} $\Leftrightarrow$ \ref{t:easy char:item3} in \cref{t:easy char fixed points on geom side}.
\end{proof}

\section{The lattice-groups fixed by the adjunction} \label{s:fixedalgebras}

The aim of this section is to prove that a unital Abelian $\ell$-group $G$ is metrically complete if and only if there is a compact a-space $X$ such that $G$ is isomorphic to $\C(X)$ as a unital $\ell$-group; equivalently, if and only if the map $\eps_G \colon G \to \C(\Max{G})$ in \cref{p:specificadj} is an isomorphism of unital $\ell$-groups. 
Recall from the Introduction that, for any unital Abelian $\ell$-group $(G, 1)$, for $x\in G$, we write $\lvert x \rvert\df x\vee (-x)$; for $n \in \N$, we write $n$ to mean the $n$-fold sum of the unit $1$ of $G$; finally, for $p,q \in \N$ with $q>0$, and for $y \in G$, we write $x \leq \frac{p}{q}y$ as a shorthand for $\lm{q}x \leq \lm{p}y$. We recall also that an $\ell$-group $G$ is \emph{Archimedean} if, for every $x,y \in G$, whenever $\lm{n}x \leq y$ holds for every $n\in\N$, we have $x \leq 0$.

\begin{definition} \label{d:ell-norm}
In any unital Abelian $\ell$-group $(G, 1)$ we define $\lnorm{\,} \colon G \to [0,+\infty)$ by setting, for any $g\in G$, 
\[
    \lnorm{g} \df \inf \left\{ \frac{p}{q} \in \Q \mid p,q \in \N, q \neq 0\text{ and }\lvert g \rvert \leq \frac{p}{q}1 \right\}.
\]
We call $\lnorm{\,}$ the \emph{seminorm on $G$ induced by the unit}. 
We set:
\begin{equation} \label{eq:defn-metric}
    \dist(x,y) \df \lnorm{x-y},
\end{equation}
and we call $\dist$ the \emph{pseudometric induced by the unit}.
Each of the implications $\lnorm{g}=0 \Rightarrow g=0$ and $\dist(x,y)=0 \Rightarrow x=y$ holds precisely when $G$ is Archimedean, in which case we call $\lnorm{\,}$ the \emph{norm induced by the unit}, and $\dist$ the \emph{metric induced by the unit}.
\end{definition}
All these notions are standard in the literature on lattice-groups; see \cite{Goodbook} for further details. 

\begin{remark} \label{r:l-norm-is-sup-norm}
    It is well known (see, e.g., \cite[Theorem 4.14]{Goodbook}) that a unital Abelian $\ell$-group is Archimedean if and only if it is isomorphic to a subgroup of $\R^{X}$ for some set $X$.
    So, if $G$ is Archimedean, its elements can be thought of as functions from a set $X$ into $\R$. Under this interpretation, $\lnorm{\,}$ becomes the \emph{sup norm}: $\lnorm{f} = \sup \{ \lvert f(x) \rvert \mid x \in X \}$.
    We tacitly use this remark whenever needed in the paper. 
\end{remark}

\begin{remark}
    Let $(G,1)$ be a unital Abelian $\ell$-group.
    For every $g \in G$, $p \in \N$, $q \in \Np$, 
    \[
        \text{if } \lvert g \rvert \leq \frac{p}{q}1 \text{ then } \lnorm{g} \leq \frac{p}{q}.
    \]
    The converse implication holds when $G$ is Archimedean.
\end{remark}

The main fact needed to achieve the characterisation described at the beginning of this section is a generalisation of the Stone-Weierstrass Theorem that takes into account denominators (\Cref{t:StWe_a-spaces} below). To prove it we need the following preliminary results.

\begin{theorem} \label{t:StWe-figo-2-options}
    Let $X$ be a compact space, let $L$ be a set of continuous functions from $X$ to $\R$ that is closed under $\lor$ and $\land$, and suppose that either $X$ or $L$ is nonempty.
    The following conditions are equivalent for any function $f \colon X \to \R$.
    \begin{enumerate}
    	\item \label{i:unif-limit} The function $f$ is a uniform limit of a sequence in $L$.
    	\item \label{i:Stone} The function $f$ is continuous and, for all $x, y \in X$ and all $\eps > 0$, there exists $g \in L$ such that $\lvert f(x) - g(x) \rvert < \eps$ and $\lvert f(y) - g(y) \rvert < \eps$.
    	\item \label{i:up-down} The function $f$ is continuous, for all $x\neq y \in X$ and all $\eps > 0$ there exists $g \in L$ such that $g(x) > f(x) - \eps$ and $g(y) < f(y) + \eps$, and for all $z \in X$ the value $f(z)$ belongs to the closure of $\{ h(z) \mid h \in L \}$.
    \end{enumerate}
\end{theorem}

\begin{proof}
    The equivalence \ref{i:unif-limit} $\Leftrightarrow$ \ref{i:Stone} is due to M.\ H.\ Stone \cite[Theorem~1]{Stone1948}.
    The implication \ref{i:Stone} $\Rightarrow$ \ref{i:up-down} is elementary.
    Let us prove the implication \ref{i:up-down} $\Rightarrow$ \ref{i:Stone}.
    Assume \ref{i:up-down}.
    In case $x = y$, condition \ref{i:Stone} holds because there exists $h \in L$ such that $\lvert f(x) - h(x) \rvert < \eps$.
    Suppose $x \neq y$.
    Then there exists $h \in L$ such that $\lvert f(x) - h(x) \rvert < \eps$, and there exists $k \in L$ such that $\lvert f(y) - k(y) \rvert < \eps$.
    We complete the proof arguing by cases.
    \begin{enumerate}[wide]
    	\item Case $k(x) \geq f(x)$ and $h(y) \geq f(y)$. It suffices to take $g = h \land k$.
    	\item Case $k(x) \leq f(x)$ and $h(y) \leq f(y)$. It suffices to take $g = h \lor k$.
    	\item Case $k(x) \leq f(x)$ and $h(y) \geq f(y)$. Since $x \neq y$, by hypothesis there exists $l \in L$ such that $l(x) > f(x) - \eps$ and $l(y) < f(y) + \eps$. It now suffices to take $g = (h \land l) \lor k$.
    	\item Case $k(x) \geq f(x)$ and $h(y) \leq f(y)$. This case is similar to the previous one. \qedhere
    \end{enumerate}
\end{proof}

\begin{proposition} \label{p:den_and_closure}
    Let $G$ be a subgroup of $\R$ containing $1$. The closure of $G$ in $\R$ with respect to the Euclidean topology coincides with the set
    \[
        \{ r \in \R \mid \den{r} \text{ divides } \lcm\{\den g\mid g\in G\} \}.
    \]
\end{proposition}

\begin{proof}
    By \cite[Lemma 4.21]{Goodbook}, either $G = \Zmod$ for a uniquely determined $n \in \Np$, or $G$ is dense. If the former holds, then $G$ is closed, and direct inspection shows that $\lcm\{\den g\mid g\in G\} = n$. Hence, for $r \in \R$, $\den{r}$ divides $n$ precisely when $r$ is rational and it may be displayed in reduced form as $\frac{p}{q}$ with $q$ a divisor of $n$; that is, precisely when $r \in \Zmod$. On the other hand, if $G$ is dense then the set of denominators of its elements is unbounded, whence $\lcm\{\den g\mid g\in G\} = 0$.
    Since every $d \in \N$ divides $0$, the proof is complete.
\end{proof}

\begin{theorem} \label{t:StWe-1}
    Let $X$ be a compact space and let $G$ be a unital $\ell$-subgroup of the set of continuous functions from $X$ to $\R$.
    Suppose that for all $x \neq y \in X$ there exists $g \in G$ such that $g(x) \neq g(y)$.
    Then a function $f \colon X \to \R$ is a uniform limit of a sequence in $G$ if and only if $f$ is continuous and, for all $x \in X$, $f(x)$ belongs to the closure of $\{ g(x) \mid g \in G \}$.
\end{theorem}

\begin{proof}
    The statement easily follows from the equivalence \ref{i:unif-limit} $\Leftrightarrow$ \ref{i:up-down} in \Cref{t:StWe-figo-2-options}. For the right-to-left implication we explicitly check that condition \ref{i:up-down} in \Cref{t:StWe-figo-2-options} is satisfied. The function $f$ is continuous and, for all $x \in X$, $f(x)$ belongs to the closure of $\{ g(x) \mid g \in G \}$ by hypothesis.
    Let $x \neq y \in X$.
    By assumption there exists $g\in G$ such that $g(x)\neq g(y)$. 
    We claim that there exist integers $m$ and $q$ such that $mg(y) + q\leq f(y)$ and $mg(x) + q \geq f(x)$. 
    Indeed, pick a rational number $\frac{a}{b}$ in the real interval whose endpoints are $g(x)$ and $g(y)$ (in whichever order they are).
    If $g(y) < \frac{a}{b} < g(x)$, pick $m \in \Z$ such that $b$ divides $m$ and
    \[
    m\geq \max \left\{ \frac{f(y)}{g(y)-\frac{a}{b}},\frac{f(x)}{g(x) - \frac{a}{b}} \right\}.
    \]
    If $g(x) < \frac{a}{b} < g(y)$, pick $m \in \Z$ such that $b$ divides $m$ and
    \[
    m \leq \min \left\{ \frac{f(y)}{g(y)-\frac{a}{b}},\frac{f(x)}{g(x) - \frac{a}{b}} \right\}.
    \]
    In both cases, set $q \coloneqq - m\frac{a}{b}$, and note that $q \in \Z$, $mg(y) + q \leq f(x)$ and $mg(x) + q \geq f(x)$.
    This proves our claim.
    Therefore, for any $\eps > 0$, $mg(x) + q> f(x) - \eps$ and $mg(y) + q < f(y) + \eps$. 
    Since $mg + q \in G$, condition \ref{i:up-down} is satisfied. 
    The proof is complete.
\end{proof}

\begin{theorem}[Stone-Weierstrass for unital Abelian $\ell$-groups]\label{t:StWe_a-spaces}
    Let $X$ be a compact a-space, let $G$ be a unital $\ell$-subgroup of $\C(X)$, and suppose that the following conditions hold.
    \begin{enumerate}
        \item \label{i:sep} For every $x \neq y \in X$, there exists $g \in G$ such that $g(x) \neq g(y)$.
        \item \label{i:pres-den} For every $x \in X$, $\zeta(x) = \den(g(x))_{g \in G}$.
    \end{enumerate} 
    Then $G$ is dense in $\C(X)$ with respect to the uniform metric.
\end{theorem}

\begin{proof}
    In light of \Cref{t:StWe-1}, it suffices to show that, for every $f \in \C(X)$, the value $f(x)$ belongs to the closure of $\{ g(x) \mid g \in G \}$. Since $G_x\df\{ g(x) \mid g \in G \}$ is a subgroup of $\R$ containing $1$, by \Cref{p:den_and_closure}, the closure of $G_x$ is
    \begin{equation}\label{eq:closure-of-G_x:1}
        \left\{ r \in \R \mid \den{r} \text{ divides } \lcm { \{ \den{g(x)} \mid g \in G \} } \right\}.
    \end{equation}
    By assumption, $\zeta(x) = \den{ ( g(x) )_{ g \in G } } = \lcm { \{ \den{g(x)} \mid g \in G \} }$.
    Hence the set in \eqref{eq:closure-of-G_x:1} coincides with
    \begin{equation}\label{eq:closure-of-G_x:2}
    \left\{ r \in \R \mid \den{r} \text{ divides } \zeta(x) \right\}.
    \end{equation}
    Since $f$ is an a-map, $\den{f(x)} \text{ divides } \zeta(x)$, and hence $f(x)$ is in the set in \eqref{eq:closure-of-G_x:2}, as was to be shown.
\end{proof}

For the proof of \Cref{t:fixedalgebras} below we need to recall the fundamental property of unital Archimedean $\ell$-groups---namely, the existence of enough morphisms to the reals.

\begin{lemma}[See e.g.\ \mbox{\cite[Theorem 4.14]{Goodbook}}] \label{l:enoughhomstoR}
    A unital $\ell$-group $G$ is Archimedean if and only if for every $0 \neq x \in G$ there exists a unital $\ell$-homomorphism $h \colon G \to \R$ such that $h(x) \neq 0$.
\end{lemma}

\begin{theorem} \label{t:fixedalgebras}
    For every unital Abelian $\ell$-group $G$, the following are equivalent.
    \begin{enumerate}
        \item \label{i:complete} $G$ is metrically complete.
        \item \label{i:iso} For some compact a-space $X$, $G$ is isomorphic to $\C(X)$ as a unital $\ell$-group.
        \item \label{i:iso-eps} The map $\eps_G \colon G \to \C(\Max{G})$ in \cref{p:specificadj} is an isomorphism of unital $\ell$-groups.
    \end{enumerate}
\end{theorem}

\begin{proof}
    The implication \ref{i:iso-eps} $\Rightarrow$ \ref{i:iso} is immediate. 
    
    For the implication \ref{i:iso} $\Rightarrow$ \ref{i:complete}, we prove that $\C(X)$ is metrically complete. It is clear that $\C(X)$ is Archimedean. A Cauchy sequence $\succustom{f_n}[n][\N]$ of elements of $\C(X)$ is also a Cauchy sequence in $\Cc(X)$, and it is classical that $\Cc(X)$ is complete in the uniform metric.  Let then $f$ be the limit of $\succustom{f_n}[n][\N]$ in $\Cc(X)$. Let $x \in X$ be such that $\zeta(x) > 0$.  We have $f_n(x) \in \Zmod[\zeta(x)]$ for any $n \in \N$ because each $f_n$ is an a-map. Further, $\frac{1}{\zeta(x)} \Z$ is closed in $\R$ and therefore $\lim_{n\to \infty}f_n(x)=f(x) \in \Zmod[\zeta(x)]$.  Thus, $f$ is an a-map. 
    
    To prove \ref{i:complete} $\Rightarrow$ \ref{i:iso-eps}, we first show $\ker{\eps_G} = \{0\}$. If $0\neq g \in G$, then by \Cref{l:enoughhomstoR} there is a unital $\ell$-homomorphism $h \colon G\to \R$ such that $h(g) \neq 0$. Hence $\ev_g(h) \neq 0$, and thus $0\neq \ev_g\in \C(\Max{G})$; therefore, $g\not \in \ker{\eps_G}$. Having established that $\eps_G$ is injective, let us identify $G$ with the image of $\eps_G$, a unital $\ell$-subgroup of $\C(\Max{G})$. Inspection of the definitions shows that conditions \ref{i:sep} and \ref{i:pres-den} in \Cref{t:StWe_a-spaces} hold. Therefore, $G$ is dense in $\C(\Max{G})$. Since $G$ is metrically complete, $G = \C(\Max{G})$.
\end{proof}

\begin{theorem} \label{t:MAIN}
    The dual adjunction in \cref{p:specificadj} descends to a duality between the category $\cG$ of metrically complete $\ell$-groups and the category $\KHz$ of normal a-spaces and a-maps between them.
\end{theorem}

\begin{proof}
    By \cref{c:fixed-spaces} and \cref{t:fixedalgebras}.
\end{proof}

As mentioned in the Introduction, Yosida duality is a consequence of \Cref{{t:MAIN}}. Let $\iota$ be the full embedding of the category $\KH$ into $\KHz$ which equips a compact Hausdorff space with the denominator function constantly equal to $0$.  On the other hand, let $U$ be the forgetful functor from the category $\overline{\mathsf{V}}$ of metrically complete vector lattices into $\cG$. 
Every lattice-group homomorphism from a vector lattice to an Archimedean vector lattice is linear \cite[Corollary 11.53]{MR1417259}; it follows that $U$ is full.
Therefore:
\begin{corollary} \label{c:Yosida}
    The duality in \cref{t:MAIN} between $\KHz$ and $\cG$ descends to Yosida duality, i.e.\ the following squares commute up to natural isomorphism.
    \[
    \begin{tikzcd}
        \KH^\op \arrow[shift left = 0.2em]{r}{\Cc} \arrow[hook, swap]{d}{\iota}& \overline{\mathsf{V}} \arrow[hook]{d}{U} \arrow[shift left = 0.2em]{l}{\Max}\\
        \KHzop \arrow[shift left = 0.2em]{r}{\C} & \cG \arrow[shift left = 0.2em]{l}{\Max}
    \end{tikzcd}
    \]
\end{corollary}

\appendix


\gdef\thesection{\Alph{section}} 
\makeatletter
\renewcommand\@seccntformat[1]{\appendixname\ \csname the#1\endcsname.\hspace{0.5em}}
\makeatother

\section{Proof of the dual adjunction} \label{s:appendix}

We show here how to obtain the adjunction of \Cref{p:specificadj} via the theory of general concrete dualities as summarised in \cite{PorstTholen} (see also \cite{DimovTholen}).
The set $\R$ plays the r\^{o}le of dualising object, but since as a space it is not compact  (and thus is not in $\Kz$), we shall need to consider wider categories. With a slight abuse of notation, we use the same symbols for functions defined previously in this article that are considered here in a more general context.

\label{d:Alg} One of the two categories partaking in the adjunction is $\Alg$. Objects of $\Alg$, known as \emph{$\tau$-algebras},  are sets equipped with the operations $\tau \coloneqq\{+,\wedge,\vee,-,0,1\}$ of respective arities $2$, $2$, $2$, $1$, $0$, $0$;    morphisms are the homomorphisms. Objects of $\Alg$ are called \emph{$\tau$-algebras}. The other category involved is $\Tz$, introduced above in Definition \ref{d:a-space}. 

On the one hand, $\R$ is a unital Abelian $\ell$-group---and hence a $\tau$-algebra---under addition and natural order, with unit $1$. On the other hand, $\R$ is an a-space when considered with its Euclidean topology, along with the denominator function $\den$ of \Cref{d:real-denominator}.

\begin{lemma}
[Initial lifts to $\Tz$]\label{l:init-lift-Tz}
	Let $X$ be a set, let $\left(Y_i,\zeta_{i}\right)_{i \in I}$ be a family of a-spaces, and, for each $i \in I$, let $f_i \colon X \to Y_i$ be a function. Endow $X$ with the topology initial with respect to $\left(f_i\right)_{i \in I}$, i.e.\ the topology generated by $f_i^{-1}[U]$,
	for $i \in I$ and $U$ an open subset of $Y_i$.   Set
	\begin{align*}
		\zeta \colon X & \longrightarrow \N\\
		x & \longmapsto \lcm{\left\{\zeta_i(f_i(x)) \mid i \in I \right\}}.
	\end{align*} 
	Then $X$ is an a-space and for every a-space $Z$ a function $g \colon Z\to X$  is an a-map if and only if, for every $i \in I$, $f_i \circ g$ is an a-map.	 
\end{lemma}
\begin{proof}
We first prove that $X$ is an a-space.
Let $n \in \N$ and let us prove that $\zeta^{-1}[\DIV n]$ is a closed subset of $X$.
We have
\begin{align*}
    \zeta^{-1}[\DIV n] & = \{x \in X \mid \zeta(x) \in \DIV n\}\\
    & = \left\{ x \in X \mid \lcm{\left\{\zeta_i(f_i(x)) \mid i \in I \right\}} \text{ divides } n \right\}\\
    & = \left\{ x \in X \mid \text{for all }i \in I,\, \zeta_i(f_i(x)) \text{ divides } n \right\}\\
    & = \bigcap_{i \in I} \left\{ x \in X \mid \zeta_i(f_i(x)) \text{ divides } n \right\}\\
    & = \bigcap_{i \in I} \left\{ x \in X \mid \zeta_i(f_i(x)) \in \DIV n \right\}\\
    & = \bigcap_{i \in I}\zeta_i^{-1}[\DIV n],
\end{align*}
which is closed because it is an intersection of closed sets. This proves that $X$ is an a-space.

Let $Z$ be an a-space, let $g \colon Z \to X$ be a function, and let us prove that $g$ is an a-map if and only if for every $i \in I$ $f_i \circ g$ is an a-map.

For the left-to-right implication it suffices to note that, for each $i \in I$, $f_i$ is an a-map by construction.
For the converse implication, $g$ is continuous because $f_i\circ g$ is continuous for each $i \in I$ and the topology on $X$ is initial with respect to the $f_i$'s. It remains to verify that $g$ decreases denominators.
For this, write $\zeta_0 \colon Z \to \N$ for the denominator function of $Z$. Given $z \in Z$, by definition
	\begin{align} \label{eq:l:init-lift-Tz1}
		\zeta(g(z)) = \lcm{\left\{\zeta_i(f_i(g(z))) \mid i \in I \right\}}.
	\end{align}
	Further, since each $f_i \circ g$ is an a-map,
	\begin{align} \label{eq:l:init-lift-Tz2}
		\zeta_i(f_i(g(z))) \text{ divides } \zeta_0(z).
	\end{align}
	Now \eqref{eq:l:init-lift-Tz2} states that $\zeta_0(z)$ is a multiple of each member of the set $\left\{\zeta_i(f_i(g(z))) \mid i \in I \right\}$, and \eqref{eq:l:init-lift-Tz1} says that $\zeta(g(z))$ is the least common multiple of this set. Hence, $\zeta(g(z)) \text{ divides } \zeta_0(z)$, as was to be shown.
\end{proof}

For any $\tau$-algebra $G$ and any $g \in G$, we define the function $\ev_g \colon \Max{G} \to \R$ by setting $\ev_g(x)\df x(g)$.
As a consequence of \Cref{l:init-lift-Tz}, for any $\tau$-algebra $G$, the set $\Max{G}$ admits an \emph{initial lift} to $\Tz$ along the family of evaluations 
\[
\left(\ev_g \colon\Max{G}\rightarrow \R\right)_{g \in G}
\]
in the sense of \cite{PorstTholen}.
In detail, the initial lift endows $\Max{G}$ with the structure induced by restriction from the inclusion $\Max{G} \seq \R^G$. Explicitly, the topology is generated by 
\[
\{x \in \Max{G}\mid x(g) \in U\},
\]
for $g \in G$ and $U$ an open subset of $\R$. The denominator function is given by
\[
\zeta(x) \df \lcm\{x(g) \mid g \in G\},
\]
for $x \in \Max{G}$.	

Given an a-space $X$ and an element $x \in X$, we define the function $\ev_x \colon \C(X) \to \R$ by setting $\ev_x(a) \df a(x)$.

\begin{lemma} \label{l:C(X)-is-subalgebra}For each a-space $X$, $\C(X)$ is a subalgebra of the product $\tau$-algebra $\R^X$.
\end{lemma}
\begin{proof}[Sketch of proof]
	To check that $\C(X)$ is closed under the operations $+$, $\wedge$, $\vee$, $-$, $0$, and $1$, two steps are needed. (i) One verifies that the aforementioned operations in $\R$ are a-maps from the appropriate power of $\R$ to $\R$ itself. Here, we regard powers of $\R$ as a-spaces using the product topology and the denominator function \eqref{eq:den}. (ii) One uses \Cref{l:init-lift-Tz} to show that the latter structure of a-space on powers $\R^I$ of $\R$ is initial with respect to the projections $\pi_i \colon \R^I\to \R$. (Alternatively, one could prove that the needed finite powers of $\R$ with the product topology and denominator function \eqref{eq:den} indeed are the products in the category $\Tz$.) Now let $f,g \in \C(X)$, and define $f+g \colon X \to \R$ pointwise. In other words, $f+g$ is defined as the composition
	\[
	X \xrightarrow{ \ \langle f, g \rangle \ } \R^2 \xrightarrow{ + } \R\,.
	\]
	Since, by (ii), $\R^2$ is equipped with the initial structure with respect to the projections, and $f$ and $g$ are a-maps, the product map $\langle f, g \rangle$ is an a-map; since, by (i), $+$ is an a-map, it follows that the composition $f+g$ is an a-map, and thus lies in $\C(X)$. Similarly for the other operations.
\end{proof}

In light of \Cref{l:C(X)-is-subalgebra}, $\C(X)$ is naturally equipped with the structure of a $\tau$-algebra by regarding it as a subalgebra of the product $\R^X$. 

\begin{lemma}[Initial lifts to $\Alg$]\label{l:init-lift-Alg}For each a-space $X$, the structure of $\tau$-algebra on $\C(X)$ given by \Cref{l:C(X)-is-subalgebra} is initial with respect to the family of evaluations $\left(\ev_x \colon\C(X)\rightarrow \R\right)_{x \in X}$.
\end{lemma}
\begin{proof}Given any $\tau$-algebra $A$ and any function $f \colon A \to \C(X)$, suppose that $\ev_x\circ f \colon A\to \R$ is a homomorphism for each $x \in X$. We prove that $f$ commutes with $+$. The argument for the remaining operations is analogous. Given $a,b \in A$ and $x \in X$, 
	\[
	\left(f(a+b) \right)(x) = \ev_x(f(a+b)) = \ev_x(f(a))+\ev_x(f(b)) = \left(f(a) \right)(x)+\left(f(b) \right)(x) \,.
	\]
	Since the above equalities hold for an arbitrary $x \in X$, we conclude $f(a+b) = f(a)+f(b)$.
	
	Conversely, given a homomorphism $f \colon A \to \C(X)$ we need to show that $\ev_x\circ f \colon A\to \R$ is a homomorphism for each $x \in X$---this is clear, because each $\ev_x \colon A\to \R$ is a homomorphism. 
\end{proof}

Combining the preceding considerations with \cite[Theorem 1.7]{PorstTholen}, the contravariant hom-functors 
\[\hom(-, \R) \colon \Tz \longrightarrow \Set\text{ and }\hom(-, \R) \colon \Alg \to \Set\] 
lift to contravariant functors
\[
\C\colon\Tz \longrightarrow \Alg
\text { and }
\Max\colon \Alg\longrightarrow \Tz
\]
and these are linked by the following result.
\begin{proposition} \label{p:basicadj} The contravariant functors $\C \colon \Tz\to \Alg$ and $\Max \colon \Alg \to \Tz$ are adjoint with units given by
\begin{align*}
    	\eta_X \colon X&\longrightarrow \Max\C(X) & \varepsilon_G\colon G & \longrightarrow \C(\Max{G})\\
	x&\longmapsto \ev_x \colon \C(X) \to \R, & g&\longmapsto \ev_g \colon \Max{G}\to \R.
\end{align*}
\end{proposition}

\begin{theorem}\label{t:restrictedbasicadj}
The dual adjunction in \Cref{p:basicadj} descends to the full subcategories $\Gr$ of $\Alg$ on unital Abelian $\ell$-groups, and  $\Kz$ of $\Tz$ on a-spaces.
\end{theorem}
\begin{proof}
    Let $(X,\zeta)$ be an a-space. Then $\C(X)$ is an Abelian $\ell$-group because $\R$ is, and $\C(X)$ is a $\tau$-subalgebra of the power $\R^X$. Assume further that $X$ is compact. Given $f \in \C(X)$, by the Extreme Value Theorem there is $n \in \N$ such that $n 1 \geq f$, where $1\colon X \to \R$ denotes the function constantly equal to $1$. Hence, $1$ is a unit of $\C(X)$. 
	
	Conversely, let $G$ be a unital Abelian $\ell$-group. We first show that $\Max{G}$, which is a subset of $\R^G$, is contained in a product of compact intervals in $\R$. Indeed, for any $g \in G$ there is $n_g \in \N$ such that $-n_g\leq g \leq n_g$. Thus, for any $h \in \hom(G,\R)$ we have $-n_g\leq h(g) \leq n_g$. Therefore $\Max{G}\subseteq \prod_{g \in G}[-n_g, n_g]$. Since $\prod_{g \in G}[-n_g, n_g]$ is compact by Tychonoff's Theorem, $\Max{G}$ is compact as soon as it is closed. To establish the latter we write $\Max{G}$ as an intersection of closed sets. Indeed, 
	\begin{align}
		\left\{f \in \R^G\mid \forall x,y \in G \ f(x+y) = f(x)+f(y) \right\}& =
		\bigcap_{x,y \in G}\left\{f \in \R^G\mid f(x+y) = f(x)+f(y) \right\} \nonumber\\
		& = \bigcap_{x,y \in G}\left\{ f \in \R^G \mid \ev_{x+y}(f) = \ev_x(f) + \ev_y(f) \right\}.\label{eq:hausdorffequaliserclosed}
	\end{align}
	Now, the evaluation $\ev_{z}$ is continuous for each $z \in G$. Moreover, since $+\colon \R^2\to\R$ is continuous, $\ev_x+\ev_y$ is continuous. Since $\R$ is Hausdorff, the sets appearing in the intersection \eqref{eq:hausdorffequaliserclosed} are closed, and hence their intersection is closed. The remaining operations $\wedge$, $\vee$, $-$, $0$, and $1$ are continuous, too, so an analogous argument shows that $\Max{G}$ is an intersection of closed sets. The proof is complete.
\end{proof}

\section*{Acknowledgements}
We are grateful to Luca Carai for pointing out the equivalence between \ref{d:anormal-space:item3prime} and \ref{d:anormal-space:item3}, which led us to adopt \ref{d:anormal-space:item3prime} in the definition of a normal a-space.
Marco Abbadini and Luca Spada were supported by the Italian Ministry of University and Research through the PRIN project n.\ 20173WKCM5 \emph{Theory and applications of resource sensitive logics}. 
Marco Abbadini was also partially supported by UK Research and Innovation (UKRI) under the UK government’s Horizon Europe funding guarantee (grant number EP/Y015029/1, Project ``DCPOS'') and by the ``National Group for Algebraic and Geometric Structures, and their Applications'' (GNSAGA -- INdAM).

\end{document}